\documentclass[a4paper,11pt,reqno]{article}

\usepackage{relsize}
\usepackage{cite}
\usepackage{color}
\usepackage[dvipsnames]{xcolor}

\usepackage{hyperref, enumitem}
\hypersetup{
  colorlinks   = true, 
  urlcolor     = blue, 
  linkcolor    = Purple, 
  citecolor   = red 
}
\usepackage{amsmath,amsthm,amssymb}

\usepackage[margin=2.5cm]{geometry}

\usepackage{stmaryrd}

\usepackage{array}
\newcolumntype{x}[1]{>{\centering\arraybackslash\hspace{0pt}}p{#1}}

\theoremstyle{definition}
\newtheorem{theorem}{Theorem}[section]
\newtheorem{definition}[theorem]{{{Definition}}}
\newtheorem{example}[theorem]{{{Example}}}
\newtheorem{notation}[theorem]{{{Notation}}}
\newtheorem{remark}[theorem]{{{Remark}}}

\newtheorem{corollary}[theorem]{{{Corollary}}}
\newtheorem{proposition}[theorem]{{{Proposition}}}
\newtheorem{lemma}[theorem]{{{Lemma}}}

\newtheorem{construction}{{{Construction}}}

\newcommand{\numberset}{\mathbb}
\newcommand{\N}{\numberset{N}}
\newcommand{\E}{\numberset{E}}

\newcommand{\F}{\numberset{F}}

\newcommand{\mS}{\mathcal{S}}
\newcommand{\mT}{\mathcal{T}}
\newcommand{\mC}{\mathcal{C}}

\newcommand{\mL}{\mathcal{L}}
\newcommand{\mA}{\mathcal{A}}

\newcommand{\mB}{\mathcal{B}}

\newcommand{\mM}{\mathcal{M}}
\newcommand{\M}{\mathcal{M}}
\newcommand{\mD}{\mathcal{D}}

\newcommand{\sH}{\sigma}

\newcommand{\Fq}{\F_q}

\DeclareMathOperator{\GL}{GL}

\DeclareMathOperator{\Var}{Var}
\DeclareMathOperator{\PG}{PG}

\newcommand{\wH}{\omega}

\numberwithin{equation}{section}
\allowdisplaybreaks

\usepackage{setspace}

\title{\textbf{Three Combinatorial Perspectives on Minimal Codes}}

\usepackage{authblk}
\author[1]{Gianira N. Alfarano\thanks{Gianira N. Alfarano  was supported by the Swiss National Science Foundation through grant no. 188430.}}
\affil[1]{Institute of Mathematics, University of Zurich, Switzerland}

\author[2]{Martino Borello}
\affil[2]{Universit\'e Paris 8, Laboratoire de G\'eom\'etrie, Analyse et Applications, LAGA,
Universit\'e Sorbonne Paris Nord, CNRS, UMR 7539, France}

\author[3]{Alessandro Neri\thanks{Alessandro Neri was supported by the Swiss National Science Foundation through grant no. 187711.}}
\affil[3]{Institute for Communication Engineering, Technical University of Munich, Germany}

\author[4]{Alberto Ravagnani}
\affil[4]{Department of Mathematics and Computer Science, Eindhoven University of Technology, the Netherlands}

\date{}

\begin{document}

\maketitle

\begin{abstract}
  We develop three approaches of combinatorial flavour to study the structure of minimal codes and cutting blocking sets in finite geometry,
  each of which has a particular application.
  The first approach uses techniques from algebraic combinatorics, 
  describing the supports in a linear code via the Alon-F\"uredi Theorem and the Combinatorial Nullstellensatz. 
  The second approach combines methods from coding theory and statistics to compare the mean and variance of the nonzero weights in a minimal code. Finally, the third approach
    regards minimal codes as cutting blocking sets and studies these using the theory of spreads in finite geometry.
        Applying and combining these approaches with each other, we derive several new  bounds and constraints on the parameters of minimal codes.
        Moreover, we obtain two new constructions of cutting blocking sets of small cardinality in finite projective spaces. In turn, these allow us to give explicit constructions
        of minimal codes having short length for the given field and dimension. 
 \end{abstract}

\ 

\bigskip

\section*{Introduction}

In a linear code, a codeword is  
\textit{minimal} if its support does not contain the support of any codeword other than its scalar multiples.
A code is \textit{minimal} if  its codewords are all minimal.

Minimal codewords in linear codes were originally
studied in connection with decoding algorithms~\cite{MR551274} and have been used by Massey~\cite{Massey} to determine the access structure in his code-based secret sharing scheme. However, describing the minimal codewords of a linear code is in general a difficult problem, even for highly structured families of codes.

General properties of the minimal codewords of a code are studied in \cite{MR1664103}, where 
a sufficient condition for a code to be minimal is presented (often called
the \textit{Ashikhmin-Barg condition}). 

The latter shows that a linear code in which the minimum and maximum weight are close enough to each other is necessarily minimal. Recently, estimates for the number of minimal codewords in a given code have been also found; see \cite{kiermaier2019minimum}.

In the last decade, minimal codes have been the subject of intense mathematical research,
yet their structural properties are far from being understood.
First results on minimal codes were presented in \cite{chabanne2013towards}, where the main motivation arises from secure two-party computation. 
Moreover, in the same paper an upper bound on the rate of a minimal code is established, which was recently improved in \cite{alfarano2019geometric}.
Other bounds on the minimum and maximum weight of minimal codes can be found in \cite{MR3163591}.

Various explicit constructions of minimal codes relying on
the aforementioned Ashikhmin-Barg condition are known; see~\cite{MR2235283, MR3352515} 
among many others. 
Constructions that exploit in other ways the minimal structure of the code are based, for example, 
on functions over finite fields \cite{sihem1,sihem2,bartoli2019minimalLin, bonini2020minimal}. A geometric approach was proposed in \cite{alfarano2019geometric,tang2019full,lu2019parameters}, where minimal codes are characterized as cutting blocking sets.

A remarkable property of minimal codes is that they form an asymptotically good family~\cite{MR3163591,alfarano2019geometric}. Since the proofs of~\cite{MR3163591,alfarano2019geometric} are nonconstructive,
this naturally poses the problem of \textit{explicitly} constructing families of minimal codes of short length for a given dimension, which is equivalent to constructing small cutting blocking set in a given projective space. Problems of this type are very natural and yet wide open challenges in the realm of extremal combinatorial structures; see e.g. \cite{blocking,ball1996multiple, barat2004multiple}. An important contribution in this direction 
is \cite{fancsali2014lines}, where the authors construct 
small cutting blocking sets in $\PG(k-1,q)$, under the assumption that the characteristic of the field is strictly greater than $k-1$ and the field size is at least $2k-3$. \label{disc} 
Because 
of the constraints imposed on the field size,
the construction of \cite{fancsali2014lines} is 
of limited applicability in coding theory and does not address 
the problem of constructing asymptotically good families of minimal codes (where $q$ is fixed and $k$ tends to infinity together with the code length).
More recently, a construction of cutting blocking sets in $\PG(3,q)$ and $\PG(5,q)$, which are smaller that the previously known ones, has been given in \cite{bartoli2020cutting}. This construction produces  minimal codes of dimension respectively 4 and 6 over a finite field of arbitrary size.

\medskip

\paragraph{Our contribution.} \, In this paper, we propose three different approaches of strong combinatorial flavour to the study of minimal codes, each of which has a particular application. Most methods apply more generally to arbitrary linear codes, but give the best and most explicit results when combined with the minimality property
of the underlying code.

The idea behind the first approach is to associate to a code
a multivariate polynomial, which we call the \textit{support polynomial}. This allows us to capture the combinatorics of the nonzero codewords of a code in an algebraic fashion, characterizing the 
inclusion relations among supports 
as the nonvanishing of a polynomial of bounded degree.  We then study the support polynomial using tools from
algebraic combinatorics, most notably the Alon-F\"uredi Theorem.
As an application of this method, we obtain new lower bounds for both the minimum distance and the length of a minimal code.
This improves on known results and excludes the existence of minimal codes for several new parameter sets.

The second approach uses instead ideas from statistics. More precisely, we regard the weight of a nonzero codeword as a discrete random variable and use Pless' equations, along with classical inequalities, to compare its mean and variance. All of this establishes inequalities between the maximum and minimum weight in a linear code, which are sharp for certain code families.
In turn, 
these yield a new upper bound for the minimum distance of a minimal code and exclude the existence of such codes for yet other parameter sets.

Finally, the third approach is based 
on the correspondence between minimal codes and cutting blocking sets in finite geometry.
We first reduce the problem of constructing short minimal codes to that of constructing cutting blocking sets of small cardinality. Then we show how to use the theory of \textit{spreads} in projective spaces to 
obtain cutting blocking sets whose parameters can be  computed explicitly.
The applications of this geometric approach are twofold: On the one hand, we obtain new explicit constructions of short minimal codes;
on the other hand, we establish a recursive upper bound for the least length of a minimal code over $\F_q$ having prescribed dimension.

For convenience of the reader we conclude the Introduction by listing the main contributions made by this paper, pointing to the corresponding statements.

\begin{itemize}\setlength\itemsep{0.3cm}
\item[---] As an application of methods from algebraic combinatorics, in particular the Alon-F\"uredi Theorem and the Combinatorial Nullstellensatz:

\vspace{-0.3cm}

\begin{enumerate}[leftmargin=1.3cm]\setlength\itemsep{0cm}
    \item a lower bound on the minimum distance of a minimal code (Theorem~\ref{thm:lower_bound_old_conjecture});
    \item a structural result on the maximal codewords in a linear code (Theorem~\ref{thm:codeword_support_intersection});
    \item a lower bound on the block length of a minimal code (Theorem~\ref{thm:lower_bound_length}).
 \end{enumerate}  
 
 \item[---] Combining ideas from coding theory and statistics with the algebraic combinatorial approach outlined above:
 
 \vspace{-0.3cm}
    
\begin{enumerate}[leftmargin=1.3cm]\setlength\itemsep{0cm}\setcounter{enumi}{3}
    \item an upper bound on the minimum distance of a minimal code and a constraint on its parameters
    (Corollary~\ref{coro:lbd});
    
    \item a result connecting the relative difference between maximum and minimum weights in a  linear code with its block length (Corollary~\ref{thm:ngrows}).
\end{enumerate}
    
\item[---] Using methods from projective geometry, most notably the theory of spreads:

\vspace{-0.3cm}
    
\begin{enumerate}[leftmargin=1.3cm]\setlength\itemsep{0cm} \setcounter{enumi}{5}
    \item a construction of cutting blocking sets from spreads in finite geometry and of the corresponding minimal codes (Theorems~\ref{thm:ConstructionPlanarSpreads} and~\ref{thm:ConstructionGeneralSpreads});
    \item an inductive construction of cutting blocking sets of small cardinality and of the corresponding minimal codes
    (Proposition~\ref{prop:inductiveConstruction} and Theorem~\ref{thm:inductiveCardinality});
    \item two new general constructions of short minimal codes (Constructions~\ref{construction_even} and~\ref{construction_baer}).
\end{enumerate}
\end{itemize}

\smallskip

\paragraph{Outline.} \  The paper is overall organized into four sections. Section~\ref{sec:1} contains the preliminaries on minimal codes and illustrates their connection with cutting blocking sets. Each of the  remaining three sections is devoted to a different approach to minimal codes, using algebraic combinatorics (Section~\ref{sec:2}), statistics (Section~\ref{sec:3}), and finite geometry (Section~\ref{sec:4}).

\bigskip

\bigskip

\bigskip

\section{Preliminaries}
\label{sec:1}

In this section we establish the terminology for the remainder of the paper 
and state some preliminary results on the parameters of minimal codes.
These will be applied in several instances in the sequel. All codes considered in this work are linear.

\begin{notation}
Throughout this paper, $q$ is a prime power, $\F_q$ is the finite field with $q$ elements, and $n$, $k$ are integers with $n \ge k \ge 1$. For $i \in \N =\{0,1,2, \ldots\}$ we let $[i]:=\{j \in \N \, : \, 1 \le j \le i\}$. We only consider row-vectors and for any  matrix $M\in \F^{a \times b}$ we denote by $\mathrm{rowsp}(M)$ the rowspace of $M$ over $\F$, that is the $\F$-subspace of $\F^b$ generated by the rows of $M$. Finally, for $i\in\N_{\ge 1}$ we denote by $e_i$ the $i$-th standard basis vector. 
\end{notation}

\subsection{Minimal Codes}

In this short subsection we define minimal codes and briefly survey some of their main properties. We will use them repeatedly throughout the paper.

\begin{definition}
The (\textbf{Hamming}) \textbf{support} of a vector 
$v \in \F_q^n$ is 
$\sH(v)=\{i \mid v_i \neq 0\} \subseteq [n]$
and its (\textbf{Hamming}) \textbf{weight} is $\wH(v)=|\sH(v)|$.

An $[n,k]_q$ \textbf{code} is a nonzero $\F_q$-linear subspace $\mC \subseteq \F_q^n$ of dimension $k$. Its elements are called \textbf{codewords}.
The \textbf{minimum distance} of
$\mC$ is the integer $d(\mC)=\min\{\wH(c) \mid c \in \mC, \, c \neq 0\}$ and its \textbf{maximum weight} is $\max\{\wH(c) \mid c \in \mC\}$. If $d=d(\mC)$ is known, we say that $\mC$ is 
an $[n,k,d]_q$ code. A \textbf{generator matrix}  $G\in\F_q^{k\times n}$ of $\mC$ is a matrix such that $\mathrm{rowsp}(G)=\mC$.

Finally, codes $\mathcal{C}$ and $\mathcal{C}^\prime$ are called (\textbf{monomially})  \textbf{equivalent} if there exists an $\F_q$-linear isometry $f: \F_q^n \to \F_q^n$ with $f(\mC)=\mC'$; see~\cite[page 24]{huffman2010fundamentals}.
\end{definition}

Recall that an $[n,k]_q$ code $\mC$ is \textbf{nondegenerate} if there is no $i~\in~[n]$ with $c_i=0$ for all $c\in \mC$. Furthermore, $\mC$ is called \textbf{projective}  if in one (and thus in all) generator matrix~$G$ of~$\mC$ no two columns are proportional. Note that a projective code is necessarily nondegenerate.

In this paper we mostly concentrate on codes whose codewords are all minimal.

\begin{definition} \label{def:mc}
Let $\mC$ be an $[n,k]_q$ code.
A nonzero codeword $c \in \mC$ is called \textbf{minimal} 
if every codeword $c' \in \mC$ with $\sigma(c') \subseteq \sigma(c)$ is a multiple of $c$. We say that $\mC$ is \textbf{minimal} if all its codewords are minimal.
\end{definition}

\begin{remark}\label{rem:maximalminimal}
Following the notation of Definition~\ref{def:mc},
in a minimal code $\mC$ any nonzero codeword $c$ is minimal, but also \textbf{maximal} (i.e., every other codeword $c' \in \mC$ with 
$\sigma(c') \supseteq \sigma(c)$ is a multiple of $c$).
\end{remark}

The following simple result states that every minimal codeword $c$ in a $[n,k]_q$ code~$\mC$ has weight upper bounded by $n-k+1$. To see this, it suffices to puncture $\mC$ on the nonzero positions of $c$, obtaining
a new code whose length is $n-\wH(c)$ and whose dimension is $k-1$.

\begin{proposition}\label{prop:weights}
 Let $\mC$ be an $[n,k]_q$ code.
  Every minimal codeword $c \in \mC$ has $\wH(c) \le n-k+1$.
\end{proposition}

The following result shows that minimal codes have relatively large length with respect to their dimension and field size; see also Remark~\ref{rrr}.

\begin{theorem}[see \cite{alfarano2019geometric, lu2019parameters}] \label{firstn}
Let $\mC$ be an $[n,k]_q$ minimal code with $k \ge 2$. We have 
$n \ge (k-1)q+1$.
\end{theorem}

The previous bound is not tight in general. More precisely, in \cite{alfarano2019geometric} it was conjectured (and then proved in \cite{tang2019full}) that the length $n$ of an $[n,k]_q$ minimal code satisfies the following lower bound.

\begin{theorem}[see \cite{alfarano2019geometric,tang2019full}] \label{ttt}
Let $\mC$ be an $[n,k]_q$ minimal code with $k \ge 2$. We have 
\begin{equation*}\label{eq:boundGriesmer} n\geq (k-1)(q-1)+1+\sum_{i=1}^{k-1}\left\lceil\frac{(k-1)(q-1)+1}{q^i}\right\rceil.\end{equation*}
\end{theorem}

In Section~\ref{sec:2} we will further improve the bound  in Theorem~\ref{ttt} using methods from algebraic combinatorics; see Theorem~\ref{thm:lower_bound_length}.

\subsection{Minimal Codes and Cutting Blocking Sets}\label{sec:mincodescut}

The concept of a \textit{cutting blocking set} was introduced in \cite{bonini2020minimal} with the goal of constructing a family of minimal codes. The same objects were know earlier under various names and in different contexts. In \cite{davydov2011linear} these are 
called $N$-\emph{fold strong blocking set} and are used for constructing small saturating sets in projective spaces over finite fields. In \cite{fancsali2014lines}, cutting blocking sets are  referred to as \emph{generator sets} and are constructed as union of disjoint lines.
In \cite{alfarano2019geometric} and \cite{tang2019full} it was independently shown that cutting blocking sets are in one to one correspondence with minimal linear codes.  In this subsection, we recall some properties of (cutting) blocking sets and known results about their size.

Consider the  finite projective geometry of dimension $N$ and order $q$, denoted by $\PG(N,q)$. Recall that
$$ \PG(N,q):= \left(\F_q^{N+1}\setminus \{0\}\right)/_\sim, $$
where $\sim$ denotes the proportionality relation, i.e.,
$u\sim v$ if and only if $u=\lambda v$
for some nonzero $\lambda \in \F_q$.
A \textbf{$d$-flat} in $\PG(N,q)$ is a subspace $\Pi$  isomorphic to $\PG(d,q)$. A $1$-flat is  a \textbf{line}, while a $2$-flat is a \textbf{plane}. If $d=N-1$, then $\Pi$
is called a \textbf{hyperplane}.

In our approach, projective systems are crucial geometric objects for the study of linear codes and their properties.

\begin{definition}\label{def:projsystem}
 A \textbf{projective} $[n,k,d]_q$ \textbf{system} $\mathcal{P}$ is a finite  set of $n$ points (counted with multiplicity) of $\PG(k-1,q)$ that do not all lie on a hyperplane and such that $$d = n- \max\{|H\cap \mathcal{P}| \, : \,  H\subseteq \PG(k-1,q), \; \dim(H) = k-2\}.$$ 
 Projective $[n,k,d]_q$ systems $\mathcal{P}$ and $\mathcal{P}^\prime$ are \textbf{equivalent} if there exists a projective isomorphism~$\phi$ of $\PG(k-1,q)$ mapping $\mathcal{P}$ to $\mathcal{P}^\prime$ which preserves the multiplicities of the points.
\end{definition}

There is a well-known correspondence\label{page:correspondence} between the (monomial) equivalence classes of nondegenerate $[n,k,d]_q$ linear codes and the equivalence classes of projective $[n,k,d]_q$ systems; see \cite[Theorem~1.1.6]{MR1186841}.
More precisely, let $G$ be a $k\times n$ generator matrix of an $[n,k]_q$ linear code. Consider the set $\mathcal{P}$ of one-dimensional subspaces of $\F_q^n$ spanned by the columns of $G$, which gives a set of points in $\PG(k-1,q)$.  Conversely, let $\mathcal{P}$ be a projective $[n,k,d]_q$ system. Choose a representative for any point of $\mathcal{P}$ and consider the code generated by the matrix having these representatives as columns.
 Now observe that for any nonzero vector $ u=(u_1,u_2,\ldots,u_k) $ in~$\F_q^k$ the  hyperplane
\[
u_1x_1+u_2x_2+\cdots + u_kx_k=0
\]
 contains $|\mathcal{P}|-w$ points of $\mathcal{P}$ if and only if the codeword $uG$ has weight $w$.

\begin{definition}Let $t, r, N$ be positive integers with $r<N$. A $t$-\textbf{fold} $r$-\textbf{blocking set} in~$\PG(N,q)$ is a set $\M\subseteq \PG(N,q)$ such that for every $(N-r)$-flat $\Lambda$ of $\PG(N,q)$ we have $|\Lambda \cap \M|\geq t$. When $r=1$, we will  refer to $\M$ as a $t$-\textbf{fold blocking set}. When $t=1$, we will refer to it as an $r$-\textbf{blocking set}. When  $r=t=1$, $\M$ is simply a \textbf{blocking set}.
\end{definition}

Cutting blocking sets are defined as follows.

\begin{definition}
 Let $r, N$ be positive integers with $r<N$. An $r$-blocking set $\M$ in $\PG(N,q)$ is \textbf{cutting} if for every pair of $(N-r)$-flats $\Lambda, \Lambda'$ of $\PG(N,q)$ we have
 $$ \M \cap \Lambda \subseteq \M \cap \Lambda' \,\, \Longleftrightarrow \,\, \Lambda =\Lambda'.$$ Equivalently, an $r$-blocking set $\mathcal{M} \subseteq \PG(N,q)$ is cutting if and only if for every $(N-r)$-dimensional subspace $\Lambda$ of $\PG(N,q)$ we have $\langle \mathcal{M} \cap \Lambda \rangle =\Lambda$; see \cite{alfarano2019geometric}.
 \end{definition}
 
It is shown in \cite{alfarano2019geometric,tang2019full} that the correspondence described above between projective $[n,k,d]_q$ systems and nondegenerate $[n,k,d]_q$ linear codes extends to a correspondence between equivalence classes of $[n,k,d]_q$ minimal codes and equivalence classes of projective $[n,k,d]_q$  systems that are  cutting blocking sets. This geometric interpretation of minimal codes will be crucial in Section~\ref{sec:4}.

\begin{remark}
As already mentioned in the Introduction, we are
particularly interested in finding lower bounds on the length of minimal codes or, equivalently, lower bounds on the size of cutting blocking sets in projective spaces. From this point of view,
it is not restrictive to only consider projective codes, which correspond to 
projective systems in which all the points have multiplicity one.
\end{remark}

It immediately follows from the definitions that a cutting blocking set $\mathcal{M}$ in $\PG(N,q)$ is necessarily an $N$-fold blocking set. The following theorem is obtained by combining a well-known result of Beutelspacher (which gives a lower bound on the cardinality of an $N$-fold blocking set in $\PG(N,q)$ when $N\leq q$) and the correspondence between minimal codes and cutting blocking sets.

\begin{theorem}[\text{see \cite[Theorem~2]{beutelspacher1983}}]\label{thm:sizeNfold}
Let $\mC$ be an $[n,k]_q$ minimal code. If $k-1\leq q$, then $n\geq (q+1)(k-1)$.
\end{theorem}

The above results uses the fact that cutting blocking sets in $\PG(k-1,q)$ are in particular $(k-1)$-fold blocking sets. Beutelspacher also characterized $(k-1)$-fold blocking sets in $\PG(k-1,q)$ with cardinality $(q+1)(k-1)$,  under the further assumption that $k\leq \sqrt{q}+2$. Recall that, when $q$ is a square, a \textbf{Baer subspace} of $\PG(N,q)$ is a subgeometry isomorphic to $\PG(N, \sqrt{q})$.

\begin{theorem}[\text{see \cite[Theorem 3]{beutelspacher1983}}]\label{thm:characterization_beutelspacher}
 Let $4\leq k\leq \sqrt{q}+2$ and let $\mathcal{M}$ be a $(k-1)$-fold blocking set in $\PG(k-1,q)$. Then $|\mathcal{M}|\geq (q+1)(k-1)$. Moreover, equality holds if and only if one of the following scenarios occurs:
 \begin{enumerate}
     \item $\mathcal{M}$ is the set of points on $k-1$ mutually skew lines.
     \item $k=\sqrt{q}+2$ and $\mathcal{M}$ is the point set of a $3$-dimensional Baer subspace of $\PG(k-1,q)$.
     \item $q=4$, $k=4$, and $\mathcal{M}$ is the complement of a hyperoval in a plane of $\PG(k-1,q)$, where an hyperoval is a set of $q + 2$ points in a plane, no three of which are collinear.
 \end{enumerate}
\end{theorem}

In \cite[Lemma 4.9]{alfarano2019geometric} and in \cite{tang2019full} it was observed that cutting blocking sets in $\PG(2,q)$ and $2$-fold blocking sets are actually the same object. Moreover, in $\PG(2,q)$ one can always construct a $2$-fold blocking set of size $3q$, or equivalently a $[3q,3]_q$ minimal code, by considering the union of three lines that do not intersect in the same point. When $q$ is a square, one can construct a cutting blocking set as union of two disjoint Baer subplanes, producing a minimal code of length $2q+2\sqrt{q}+2$. 
We thus survey the known results on the cardinality of $2$-fold blocking sets in $\PG(2,q)$, which turn out to be an accurate estimates also for the length of minimal codes of dimension~$3$.
 
 \begin{theorem}[\text{see \cite[Theorem 3.1]{ball1996size}}]\label{thm:2fold}
  Let $\mathcal{M}$ be a $2$-fold blocking set in $\PG(2,q)$. The following  hold.
  \begin{enumerate}
      \item If $q<9$, then $|\mathcal{M}|\geq 3q$.
      \item If $q>4$ is a square, then $|\mathcal{M}|\geq 2q+2\sqrt{q}+2$.
      \item\label{part:qnonsquare} If $q>19$, $q=p^{2d+1}$, then $|\mathcal{M}|\geq 2q+p^d\left\lceil\frac{(p^{d+1}+1)}{(p^d+1)}\right\rceil+2$.
      \item \label{part:qprime} If $q=11,13,17,19$ is not a square, then $|\mathcal{M}|\geq\frac{(5q+7)}{2}$.
  \end{enumerate}
 \end{theorem}

The bounds in Theorem~\ref{thm:2fold}, parts \eqref{part:qnonsquare} and \eqref{part:qprime}, are believed not to be sharp; see \cite[page 133]{ball1996size}. In particular, we are not aware of any
construction of $2$-fold blocking sets achieving these sizes.

\medskip

\begin{remark} \label{remdisc} In the literature, there are 
two \textit{general} constructions of small cutting blocking sets
we are aware of, which we briefly sketch in this remark. 

The first one was proposed by Fancsali and Sziklai in \cite{fancsali2014lines} and it works as follows. One chooses any $2k-3$ distinct points on the rational normal curve in $\PG(k-1,q)$ and takes the union of the tangent lines at these points. The resulting set is a cutting blocking set, under the  assumption that the characteristic of the field is at least $k$. We call this set the \textbf{rational normal tangent set}. The corresponding codes are minimal $[(2k-3)(q+1),k]_q$ codes whose minimum distance was proved to be at least $kq$ in \cite{bartoli2020cutting}. The drawback of this construction is the constraint on both the size ($q$) and the characteristic ($p$) of the underlying field, reading $q\geq 2k-3$ and $p\geq k$. For a fixed value of $q$, the approach of \cite{fancsali2014lines}  constructs cutting blocking sets in $\PG(k-1,q)$ for only a finite number of values of $k$.

A second construction that instead works for every choice of the parameters $k$ and $q$ can be found in~\cite{alfarano2019geometric,lu2019parameters,bartoli2019inductive}. Consider $k$ points $P_1,\dots, P_k$ in general position in $\PG(k-1,q)$ and let $\ell_{i,j}:=\langle P_i, P_j \rangle$. Then the union of these lines gives a cutting blocking set. From this construction, called \textbf{tetrahedron}, \label{tetra} one obtains a family of $[(q-1)\binom{k}{2}  +k,k,(q-1)(k-1)+1]_q$ minimal codes. As a consequence of Theorem~\ref{thm:2fold}, when $k=3$ this construction provides a minimal $2$-fold blocking set in $\PG(2,q)$ for any $q<9$.
\end{remark}

\bigskip

\section{Algebraic Combinatorial Approach}
\label{sec:2}

This section develops an algebraic combinatorial approach to study minimal codes. The method uses a generator matrix of a  linear code to build a multivariate polynomial ``machinery''. This allows us 
to study the maximal codewords of a code by applying classical results on the number of roots of  multivariate polynomials over finite grids. As an application of our method, with the aid of Alon's Combinatorial Nullstellensatz \cite{alon2001combinatorial} and the  Alon-F\"uredi Theorem \cite{alon1993covering}, we improve  known lower bounds on the minimum distance and the length of minimal codes.

It is interesting to observe that the results contained in this section are mainly exploiting the fact that in a minimal code all the codewords are \textit{maximal}, as already observed in Remark~\ref{rem:maximalminimal}. Although in a minimal code this code property is equivalent to all codewords being minimal, the focus on maximal codewords is crucial for deriving both the lower bound on the minimum distance (Theorem~\ref{thm:lower_bound_old_conjecture}) and the lower bound on the length (Theorem~\ref{thm:lower_bound_length}) of minimal codes.

\subsection{Combinatorial Nullstellensatz and Alon-F\"uredi Theorem}
We start by surveying tools from algebraic combinatorics that will be applied repeatedly. Among these are Alon's Combinatorial Nullstellensatz and the Alon-F\"uredi Theorem.

\begin{notation}
We state the results of this subsection and of the next one for an arbitrary field~$\F$.
In Subsection~\ref{comb_md} we will resume focusing on the case $\F=\F_q$ and on linear codes.
\end{notation}

For a multivariate polynomial $p \in \F[x_1,\ldots,x_k]$ and a subset $A\subseteq \F^k$, denote by $V_A(p)$ the set of zeros of $p$ in $A$, and by $U_A(p)$ the nonzeros of $p$ in $A$, i.e.,
\begin{align*} 
V_A(p) & = \left\{ v \in A\mid p(v) = 0\right\},\\ 
U_A(p) & = \left\{ u \in A\mid p(u) \neq 0\right\}.
     \end{align*}

The Alon--F\"uredi Theorem \cite[Theorem 5]{alon1993covering} gives a lower bound on the cardinality of $U_A(p)$ when~$A$ is a finite grid and $p$ is not identically zero on $A$. Equivalently, it provides an upper bound on the number of zeros of $p$. We recall it for convenience of the reader. 

\begin{theorem}[Alon--F{\"u}redi  Theorem \cite{alon1993covering}]\label{thm:AlonFuredi}
Let $A=A_1\times \ldots \times A_k\subseteq \F^k$ be a finite grid with $A_i\subseteq \F$ and $|A_i|=n_i$, where $n_1 \geq n_2 \geq \ldots \geq n_k \geq
  2$. Let $p \in \F[x_1,\ldots,x_k]$ be a polynomial that is not identically $0$ on $A$, and let $\bar{p}$ be the polynomial $p$ modulo the ideal $(f_1(x_1),\ldots, f_k(x_k))$, where $f_i(x_i)=\prod_{a\in A_i}(x_i-a)$. Then
$$|U_A(p)| \geq (n_s-\ell)\prod_{i=1}^{s-1} n_i,$$ 
  where $\ell$ and $s$ are the unique integers satisfying
  $\deg \bar{p}=\sum_{i=s+1}^k (n_i-1) + \ell$, with $1\leq s \leq k$ and $1 \leq \ell \leq n_s-1$.

\end{theorem}

The above theorem relies on the fact that the polynomial $p$ is not identically zero on the finite grid we are interested in. However, when dealing with polynomials that are not explicitly given, this property is not always easy to verify. In this direction, the celebrated Alon's Combinatorial Nullstellensatz
helps determining a sufficient condition for a polynomial to be nonzero on a finite grid. We state it here for completeness.

\begin{theorem}[Combinatorial Nullstellensatz \cite{alon2001combinatorial}]\label{thm:combinatorial_nullstellensatz}
 Let $p \in \F[x_1,\ldots, x_k]$ and let $\deg p=\sum_{i=1}^k r_i$, for some $r_1,\ldots, r_k \in \N$. Suppose that the coefficient of the monomial $x_1^{r_1}x_2^{r_2}\cdots x_k^{r_k}$ in $p$ is nonzero. Let $A:=A_1\times \ldots\times A_k \subseteq \F^k$ be a grid with $|A_i|\geq r_i+1$ for all $i \in [k]$. Then $U_A(p)\neq \emptyset$.
\end{theorem}

\subsection{The Support Polynomials}

 We denote by $\smash{g^{(i)}}$
 the $i$-th column vector of
 a matrix $\smash{G\in\F^{k\times n}}$. Moreover, we consider the vector $\smash{x=(x_1,\ldots, x_k)}$ whose entries are algebraically independent variables over $\F$.

\begin{definition}
The \textbf{support polynomial} associated with
a matrix $G\in \F^{k\times n}$ and a subset $I\subseteq [n]$ is
$$ p_{G,I}(x):=\prod_{i\in I} x \cdot g^{(i)} \in \F[x_1,\ldots,x_k].$$
\end{definition}

In our approach, support polynomials are crucial for the study of minimal codes (taking as~$G$ a generator matrix of an $[n,k,d]_q$ code and as $I$ a subset of a codeword's support). However, for the moment we  focus on general properties of support polynomials that do not necessarily arise from codes. The following result is straightforward and its proof is omitted.

\begin{proposition}\label{prop:properties_polys}
 Let $G\in \F^{k\times n}$ and $I\subseteq [n]$.
 \begin{enumerate}
     \item For every $A \in \GL(k,\F)$ 
      $$p_{AG,I}(x)=p_{G,I}(xA)=(p_{G,I}\circ L_A)(x),$$
      where $L_A$ denotes the linear map associated to the matrix $A$, that is $v \longmapsto vA$.
      \item For every $\tau \in \mathcal S_n$
      $$p_{G,\tau(I)}(x)=p_{GP_\tau,I}(x),$$
      where $P_\tau$ is the permutation matrix associated to $\tau$, such that  $$(v_1,\ldots,v_n)P_\tau=\left(v_{\tau(1)},\ldots,v_{\tau(n)}\right).$$
      \item For every $v \in \F^n$ 
      $$p_{GD_v,I}(x)=\Big(\prod_{i\in I} v_i\Big)p_{G,I}(x),$$
      where $D_v$ denotes the diagonal matrix whose diagonal is $v$.
 \end{enumerate}
 
\end{proposition}

We now study a support polynomial in connection with the rowspace of the matrix $G\in \F^{k\times n}$ that defines it. 
We first show how the zeros and nonzeros of support polynomials are related when we choose   matrices with the same rowspace.

 Let $G_1, G_2\in \F^{k\times n}$ be two  matrices such that $\mathrm{rowsp}(G_1)=\mathrm{rowsp}(G_2)$. It is easy to see that there exists $A\in\GL(k,\F)$ such that
 \begin{align*} U_{\Fq^k}(p_{G_1,I}) & =U_{\Fq^k}(p_{G_2,I})\cdot A:=\left\{uA \mid u \in U_{\Fq^k}(p_{G_1,I}) \right\}, \\
 V_{\Fq^k}(p_{G_1,I}) & =V_{\Fq^k}(p_{G_2,I})\cdot A:=\left\{vA \mid v \in V_{\Fq^k}(p_{G_1,I}) \right\}.
 \end{align*}
 Indeed, any matrix $A$ with $G_2=AG_1$ satisfies the desired properties. Moreover, the nonzeros of a support polynomial are closely related to the support of vectors belonging to the rowspace of the defining matrix. This is shown by the following simple result, whose proof is omitted.

\begin{lemma}\label{lem:nonzeros_poly_codewords}
 Let $G\in \F^{k\times n}$ be a matrix. For all
 $I\subseteq [n]$ we have
 $$ U_{\F^k}(p_{G,I})=\left\{u \in \F^k \mid \sH(u G)\supseteq I \right\}.$$
In particular,  $U_{\F^k}(p_{G,I})\neq \emptyset$ if and only if there exists $c\in \mathrm{rowsp}(G)$ such that $\sH(c)\supseteq I$. 
\end{lemma}

\subsection{Minimum Distance of Minimal Codes} \label{comb_md}

In this subsection we investigate the support polynomials of generator matrices of linear codes and their set of zeros. As a corollary of our results, we establish\footnote{While preparing the final version of this manuscript, we realized that the same conjecture has been also established in a recent preprint \cite{tang2019full}, using different methods. The approach developed in this paper also serves to describe the structure of codes that are not necessarily minimal, proving general properties of their maximal codewords.} a conjecture from \cite{alfarano2019geometric}.

We start with the following lemma, whose proof 
directly follows from Lemma~\ref{lem:nonzeros_poly_codewords} and Remark~\ref{rem:maximalminimal}.

\begin{lemma}\label{lem:nonzeros_minimal}
 Let $G \in \Fq^{k\times n}$ be a generator matrix of  an $[n,k]_q$  code $\mC$. Let $c =uG$ be a maximal codeword of $\mC$ and $I:=\sH(c)$. Then
 $$ U_{\Fq^k}(p_{G,I})=\{ \lambda u \mid \lambda \in \Fq^*\}.$$
 In particular, if $\mC$ is a minimal code then the above statement holds for every nonzero codeword.
\end{lemma}

\begin{theorem}\label{thm:lower_bound_old_conjecture}
 Let $\mC$ be an $[n,k,d]_q$ code, and let $c$ be a maximal codeword. Then $\wH(c) \geq (q-1)(k-1)+1$. In particular, if $\mC$ is minimal then $d\geq (q-1)(k-1)+1$.
\end{theorem}

\begin{proof}
 Let $c=(c_1,\ldots,c_n)\in\mC$ be a maximal codeword of weight $w$ and let $I:=\sH(c)$, i.e., $c_i\in \Fq^*$ if and only if $i\in I$. Take a generator matrix $G\in \Fq^{k\times n}$ for $\mC$ and consider the polynomial $p_{G,I}(x)\in\Fq[x_1,\ldots,x_k]$.
 Observe that  $p_{G,I}$  does not vanish identically on $\Fq^k$. Indeed, let  $u \in \Fq^k$ be the vector such that $u G=c$. Then $p_{G,I}(u)=\prod_{i\in I}c_i\neq 0$. This also ensures that $\deg p_{G,I}=w$. 
 Since $c$ is a maximal codeword, by Lemma~\ref{lem:nonzeros_minimal} we have $U_{\Fq^k}(p_{G,I})=\{\alpha \lambda \mid \alpha \in \Fq^*\}$,  which has cardinality $q-1$. 
 
On the other hand, let $\bar{p}_{G,I}$ denote the reduction of the polynomial $p_{G,I}$ modulo the ideal $(\{x_i^q-x_i \mid i \in [k]\})$. By Theorem~\ref{thm:AlonFuredi} we have
$$ |U_{\Fq^k}(p_{G,I})|\geq(q-\ell)q^{s-1},$$
where $\ell$ and $s$ are the unique integers satisfying
  $\deg \bar{p}_{G,I}=(q-1)(k-s) + \ell$, with $1\leq s \leq k$ and $1 \leq \ell \leq q-1$.
  
  Thus, combining this with the exact value of $|U_{\Fq^k}(p_{G,I})|$, we obtain 
  $q-1=|U_{\Fq^k}(p_{G,I})|\geq (q-\ell)q^{s-1}$, from which we deduce
  $s=1$.  
  Therefore, 
  \begin{equation*}
      w=\deg p_{G,I}\geq \deg\bar{p}_{G,I}=(q-1)(k-1)+\ell \geq (q-1)(k-1)+1. \qedhere
  \end{equation*}
\end{proof}

\begin{remark}
The Alon-F\"uredi Theorem (Theorem~\ref{thm:AlonFuredi}) gives a lower bound on the number of nonzeros of a multivariate polynomial in a finite grid in terms of the degree of the polynomial and the size of the grid. This result has been used in coding theory for deriving the minimum distance of generalized Reed-Muller codes; see e.g. \cite{geil2013weighted, lopez2014affine}. 
It is interesting to observe that in our Theorem~\ref{thm:lower_bound_old_conjecture}
the Alon-F\"uredi Theorem is applied in the ``opposite'' direction, i.e., 
we use it to derive a lower bound on the degree of the support polynomial associated to a maximal codeword, knowing the number of its nonzeros.
\end{remark}

\subsection{Maximal Codewords in Linear Codes}

In this subsection
we use support polynomials 
to study the structure of maximal codewords in a linear code
$\mC$. In particular, we show that for any maximal codeword $c \in \mC$ 
there exist several codewords whose support contains a large subset of the support of $c$. This property will be crucial for deriving a lower bound on the length of minimal codes in Subsection~\ref{sub:len}.

For $v=(v_1,\ldots,v_k) \in \Fq^k$, define 
$$f_v(x):=\prod_{i=1}^k\Big(\prod_{s \in \Fq \setminus \{v_i\}}(x_i-s)\Big).$$
Next, consider the ideal $I_q:=(x_1^q-x_1,\ldots, x_k^q-x_k)$ and denote by $\bar{f}$ the reduction of a polynomial $f\in \Fq[x_1,\ldots, x_k]$ modulo $I_q$. It is easy to check that for every $v \in \Fq^k$ we have $\bar{f_v}=f_v$. One can also easily prove that the set 
$\{f_v \mid v \in \Fq^k\}$
is an $\Fq$-basis for the space $\Fq[x_1,\ldots,x_k]/I_q$. Moreover, regarding the polynomials $f_v$'s as maps from $\Fq^k$ to $\Fq$, the set $\{f_v \mid v \in \Fq^k\}$ is an $\Fq$-basis of $\{\varphi:\Fq^k \longrightarrow \Fq \}$. This is due to the following well-known result.

\begin{proposition}
The evaluation map on $\Fq[x_1,\ldots,x_k]$ induces the isomorphism of $\Fq$-vector spaces 
\begin{equation}\label{eq:isomorphism} \Fq[x_1,\ldots,x_k]/I_q \cong \{\varphi:\Fq^k \longrightarrow \Fq \}. \end{equation}
In particular, for every  $p \in \Fq[x_1,\ldots,x_k]$ there exist unique $\mu_v \in \Fq$ for $v \in U_{\Fq^k}(p)$, such that 
$$ \bar{p}=\sum_{v \in U_{\Fq^k}(p)} \mu_v f_v.$$
\end{proposition}

\begin{proposition}\label{prop:polys_form}
 Let $\mC$ be an $[n,k]_q$  code and let $c=(c_1,\ldots, c_n) \in \mC$ be a maximal codeword of $\mC$ with weight $w$ and support $I:=\sH(c)$. Let $w_1$ be the unique integer in $[q-1]$ such that $w_1 \equiv w \mod (q-1)$. Then, for any $A \in \GL(k,q)$ such that the first row of $A^{-1}G$ is equal to $c$, we have
 $\bar{p}_{G,I}(x)=p_c(xA),$
 where 
 $$p_c(x)=\Big(\prod_{i=1}^wc_i\Big)x_1^{w_1}\prod_{i=2}^k(1-x_i^{q-1}).$$
\end{proposition}

\begin{proof}
 We first prove the statement in the case where the first row of $G$ is equal to $c$. 
 Observe that $\bar{p}_c=p_c$, that is, the polynomial $p_c$ is already reduced modulo $I_q$. Therefore, by the  isomorphism given in \eqref{eq:isomorphism}, we only need to show that $p_{G,I}(v)=p_c(v)$ for every $v \in \Fq^k$. By definition of $p_c$ we have
$$p_c(v)=\begin{cases} \lambda^{w_1}\prod\limits_{i=1}^w c_i & \mbox{ if } v=\lambda e_1, \\
0 & \mbox{ otherwise.}\end{cases}$$
On the other hand, by the choice of $G$ and Lemma~\ref{lem:nonzeros_minimal} we have $$p_{G,I}(\lambda e_1)=\prod_{i=1}^w(\lambda c_i)=\lambda^{w_1}\prod_{i=1}^w c_i, $$
where the last inequality follows using  the identity $\lambda^q=\lambda$. Moreover,
 $p_{G,I}(v)=0$ for every $v \notin \{\lambda e_1 \mid \lambda \in \Fq^*\}$.

The general case follows from the previous one.  We first transform  $G$ into $A^{-1}G$, where the first row of $A^{-1}G$ is equal to $c$. This implies that $p_{A^{-1}G,I}(x)=p_c(x)$. Then, using 
Proposition~\ref{prop:properties_polys}, we find  $p_{G,I}(x)=p_{A^{-1}G,I}(xA)=p_c(xA).$
\end{proof}

\begin{notation}
In the remainder of the section
we write  $x=(x_1,\ldots,x_k)$ and 
for $\alpha=(\alpha_1,\ldots, \alpha_k) \in \N^k$
we denote by $x^\alpha$ the monomial $x_1^{\alpha_1}\cdots x_k^{\alpha_k}$. Moreover, we let $\|\alpha\|:=\alpha_1+\ldots+\alpha_k$.  Finally, for a polynomial $p(x) \in \Fq[x_1,\ldots,x_k]$ and a monomial $x^{\alpha}$, we denote by $[x^{\alpha}]p(x)$ the coefficient of the monomial $x^{\alpha}=x_1^{a_1}\cdots x_k^{a_k}$ in $p(x)$.
\end{notation}

The following result on maximal codewords will be crucial in the next subsection for deriving a lower bound on the length of a minimal code.

\begin{theorem}\label{thm:codeword_support_intersection}
 Let $\mC$ be an $[n,k]_q$  code and let $c=(c_1,\ldots, c_n) \in \mC$ be a maximal codeword. 
 For every $j \in \sH(c)$ there exist $I_j \subseteq \sH(c)\setminus\{j\}$ of cardinality $(q-1)(k-1)$ and  a codeword $z \in \mC$ such that $\sH(z) \cap \sH(c) \supseteq I_j$.
\end{theorem}

\begin{proof}
 Let $c\in\mC$ be a nonzero codeword with support $I:=\sH(c)$ and weight $w=(q-1)(k-1)+w_1$. By Theorem~\ref{thm:lower_bound_old_conjecture} we have $w_1\geq 1$.
 
Assume first that $w\leq (q-1)k$, which implies $1\leq w_1 \leq q-1$. We choose a generator 
 matrix~$G$ for $\mC$ whose first row is equal to $c$, and assume that $c_i=1$ for every $i \in I$. This can be done without loss of generality, up to replacing the code with an equivalent one. By Proposition~\ref{prop:polys_form} we have
 $$\bar{p}_{G,I}(x)=p_c(x)=x_1^{w_1}\prod_{i=2}^k(1-x_i^{q-1}).$$
 Let $j\in I$ and assume that the $j$-th column of $G$ is $(1,0,\ldots,0)^\top$. Define $\mL_{I,j}:=\{L\subseteq I : j\notin~L, \,  |L|=(q-1)(k-1) \}$ and $\beta:=(w_1,q-1,q-1,\ldots,q-1)$.
 We have
 \begin{align*} (-1)^{k-1} &=[x^{\beta}]\bar{p}_{G,I}(x) \\
 &=[x^{\beta}]p_{G,I}(x)\\
 &=\sum_{L\in \mL_{I,j}}[x_2^{q-1}\cdots x_k^{q-1}]p_{G,L}(x).
 \end{align*}
 The first equality follows from direct inspection of $p_c(x)$. The second equality is due to the fact that the degree $p_{G,I}$ is equal to $(q-1)(k-1)+w_1$, which is also the degree of $\bar{p}_{G,I}$. The third equality follows  from the fact that the coefficients of  $x_1$ in the matrix $G$ are all equal to~$1$.
 Therefore, there exists $I_j\in\mL_{I,j}$ such that $[x_2^{q-1}\cdots x_k^{q-1}]p_{G,I_j}(x)\neq 0$. Let $x':=(x_2,\ldots,x_k)$ and consider the polynomial $f(x'):=p_{G,I_j}(0,x_2,\ldots,x_k)$. This polynomial has degree $|I_j|=(q-1)(k-1)$ and $[x_2^{q-1}\cdots x_k^{q-1}]f(x')=[x_2^{q-1}\cdots x_k^{q-1}]p_{G,I_j}(x)\neq 0$. Hence, by Theorem~\ref{thm:combinatorial_nullstellensatz}, there exists $v\in\Fq^{k-1}$ such that $f(v)=p_{G,I_j}(0,v) \neq 0$. By Lemma~\ref{lem:nonzeros_poly_codewords}, this implies that the codeword $z:=(0,v)G\in\mC$ satisfies $\sH(z)\supseteq I_j$.

 Now assume that $w>(q-1)k$, from which $w_1\geq q$. Let us write $w_1=a(q-1)+b$ with $1\leq b \leq q-1$. Since $w>(q-1)k$, we have $a\geq 1$. 
 Denote the vector $\beta:=(b,q-1,q-1,\ldots,q-1)$. Consider the set $T:=\{\alpha \in \N^k : \|\alpha\|=a(q-1), \alpha_i \equiv 0 \mod (q-1) \mbox{ for every } i\in[k] \}$. Then
  \begin{align*} (-1)^{k-1} &=[x^\beta]\bar{p}_{G,I}(x) \\
 &= \sum_{\alpha\in T}[x^{\beta+\alpha}]p_{G,I}(x).
 \end{align*}
 This means that there exists  $\gamma \in T$ such that $[x^{\beta+\gamma}]p_{G,I}(x)\neq 0$. 
 Define $\mL_{I,j}^{(\gamma)}:=\{L \subseteq I : |L|=w-b-\gamma_1,j\notin L\}$, $\gamma':=(0,\gamma_2,\ldots,\gamma_k)$, and $\beta':=(0,q-1,\ldots,q-1)$.
We have  
$$[x^{\beta+\gamma}]p_{G,I}(x)=\sum_{L\in \mL_{I,j}^{(\gamma)}}[x^{\beta'+\gamma'}]p_{G,L}(x)$$
and there exists $K\in\mL_{I,j}^{(\gamma)}$ such that $[x^{\beta'+\gamma'}]p_{G,K}(x)\neq 0$.
At this point we can consider the set $\mathcal X_K:=\{M\subseteq K \, : \, |M|=(q-1)(k-1) \}$ and write 
 $$[x^{\beta'+\gamma'}]p_{G,K}(x)=\sum_{M\in \mathcal X_K}\lambda_M\big([x^{\beta'}]p_{G,M}(x)\big)$$
 for some $\lambda_M\in \Fq$. 
Since this sum is nonzero, there exists $M$ such that $[x^{\beta'}]p_{G,M}(x)\neq 0$. As in the previous case, we use Theorem ~\ref{thm:combinatorial_nullstellensatz} and Lemma~\ref{lem:nonzeros_poly_codewords} to deduce that there exists a codeword $z \in \mC$ such that $\sH(z)\supseteq M$.
\end{proof}

\subsection{The Length of Minimal Codes}
\label{sub:len}

As an application of Theorem~\ref{thm:codeword_support_intersection}, we derive the following lower bound on the length of a minimal code.

\begin{theorem}\label{thm:lower_bound_length}
 Let $\mC$ be an $[n,k,d]_q$ minimal code. We have $n \geq (q+1)(k-1)$.
\end{theorem}

\begin{proof}
  Let $c \in \mC$ be a codeword of minimum weight $d$ with support $I:=\sH(c)$. Up to considering an equivalent code, we can assume $c_i=1$ for every $i \in I$. Since $c$ is in particular a maximal codeword of $\mC$, by Theorem~\ref{thm:codeword_support_intersection} there exists a codeword $z \in \mC$ such that $|I \cap \sH(z)|\geq (q-1)(k-1)$. Let $J:=I\cap \sH(z)$ and for every $\lambda \in \Fq^*$ define $J_{\lambda}:=\{j \in J \mid z_j=\lambda\}$. Clearly, $\smash{J=\bigcup_{\lambda \in \Fq^*} J_\lambda}$ and the union is disjoint. Thus by generalized pigeonhole principle there exists $\lambda' \in \Fq^*$ such that $$|J_{\lambda^\prime}|\geq \left\lceil\frac{|J|}{q-1}\right\rceil\geq  k-1.$$
  Now consider the codeword $z-\lambda'c$. Its support is 
  $\sH(z-\lambda'c)=(\sH(c)\cup \sH(z))\setminus J_{\lambda'}$ and 
  \begin{align*}\wH(z-\lambda'c)&=|\sH(c)|+|\sH(z)|-|J|-|J_{\lambda'}| \\
  & \leq d+\wH(z)-q(k-1).
  \end{align*}
  Combining this with $\wH(z-\lambda'c)\geq d$ we obtain $\wH(z) \geq q(k-1)$. Furthermore, by 
  Proposition~\ref{prop:weights} we have $\wH(z)\leq n-k+1$, from which we finally obtain $n \geq (q+1)(k-1)$.
\end{proof}

\begin{remark}
The lower bound of Theorem~\ref{thm:lower_bound_length} is an improvement on the bound in Theorem~\ref{ttt}.
Indeed, we have  
\begin{equation} \label{toprove}
(q+1)(k-1) \geq \sum_{i=0}^{k-1} \left\lceil \frac{(q-1)(k-1)+1}{q^i}\right\rceil.
\end{equation}
We do not go into the details of the proof.
\end{remark}

\begin{remark}
  Observe that Theorem~\ref{thm:lower_bound_length} is an improvement on the bound of 
  Theorem~\ref{thm:sizeNfold},  since it does not require the extra assumption that $k\leq q+1$. 
\end{remark}

We conclude this section with a detailed example building on \cite[Example~5.11]{alfarano2019geometric}.

\begin{example}
  We fix $q=3$, $k=4$ and take the minimal $[14,4,7]_3$ code $\mC$ whose generator matrix is
$$G:=\left( \begin{array}{cccccccccccccc}
0 & 0 & 0 & 0 & 0 & 0 & 0 & 1 & 2 & 1 & 1 & 1 & 2 & 1\\
1 & 1 & 1 & 0 & 0 & 0 & 0 & 0 & 0 & 0 & 1 & 2 & 1 & 1\\
0 & 1 & 2 & 1 & 1 & 1 & 0 & 0 & 0 & 0 & 0 & 0 & 2 & 2\\
0 & 0 & 0 & 0 & 1 & 2 & 1 & 1 & 1 & 0 & 0 & 0 & 1 & 2\\
\end{array} \right).$$
Let $I:=\{8,9,\ldots, 14\}$ be the support of the codeword $c$ given by the first row of $G$ and let $x=(x_1,x_2,x_3,x_4)$. We compute the associated support polynomial
$$ p_{G,I}(x)= x_1(x_3^2-x_1^2)(x_1^2-x_4^2)((x_4-x_2)^2-(x_3-x_1)^2).$$
An easy calculation shows that the reduction of $p_{G,I}(x)$ modulo $I_3=(x_1^3-x_1, x_2^3-x_2, x_3^3-x_3, x_4^3-x_4)$ is
$$ \bar{p}_{G,I}(x)=x_1(1-x_2^2)(1-x_3^2)(1-x_4^2)=p_c(x),$$
as we can also deduce from Proposition~\ref{prop:polys_form}. 

Moreover, since $\mC$ is a minimal code, we can actually see that for every $j \in I$ there exists a codeword $z^{(j)} \in \mC$ such that $\sH(z^{(j)})\supseteq I\setminus \{j\}$, as stated in Theorem~\ref{thm:codeword_support_intersection}. These $7$ codewords (up to their nonzero scalar multiples) are 
\begin{align*}
    z^{(8)} & =  (0, 0, 0, 0, 2, 1, 2, 0, 1, 1, 1, 1, 1, 2), \\
    z^{(9)} & =  (0, 1, 2, 1, 2, 0, 1, 2, 0, 1, 1, 1, 2, 2), \\
    z^{(10)} & = (1, 2, 0, 1, 2, 0, 1, 1, 1, 0, 1, 2, 1, 2), \\
    z^{(11)} & = (2, 1, 0, 2, 2, 2, 0, 1, 2, 1, 0, 2, 2, 2), \\
    z^{(12)} & = (1, 2, 0, 1, 1, 1, 0, 1, 2, 1, 2, 0, 2, 1), \\
    z^{(13)} & = (0, 2, 1, 2, 2, 2, 0, 1, 2, 1, 1, 1, 0, 2), \\
    z^{(14)} & = (0, 1, 2, 1, 1, 1, 0, 1, 2, 1, 1, 1, 1, 0). \\
\end{align*}

Finally, note that, in order to derive the lower bound on the length of minimal codes given in Theorem~\ref{thm:lower_bound_length}, we use in its proof 
that each of the codewords $z^{(j)}$ has weight at least $q(k-1)=9$. However, in this case only $z^{(8)}$ has weight $9$, while all the other codewords have weight $11$.
\end{example}

\begin{remark}\label{rem:nonsharp_bound}
 It is natural to ask whether the bound of Theorem~\ref{thm:lower_bound_length} is sharp or not. As stated in Subsection~\ref{sec:mincodescut}, minimal codes of dimension $k$ over $\Fq$ correspond to cutting blocking sets in $\PG(k-1,q)$ and a cutting blocking set is in particular a $(k-1)$-fold blocking set. 
 When we restrict to the case $4 \leq k\leq \sqrt{q}+2$, Theorem~\ref{thm:characterization_beutelspacher} characterizes a $(k-1)$-fold blocking set~$\mM$ in $\PG(k-1,q)$ of cardinality $(q+1)(k-1)$. This only happens in three cases.
 
 \noindent \underline{Case I:} $\mM$ is the  union of $k-1$ disjoint lines. In this case $\mM$ cannot be a cutting blocking set. To see this, write $\mM=\ell_1\cup\ldots\cup \ell_{k-1}$. Pick $P_1\in \ell_1,\ldots, P_{k-1} \in \ell_{k-1}$ and let $\Lambda:=\langle P_1,\ldots, P_{k-1}\rangle$. If $\dim (\Lambda)\leq k-3$, then $\Lambda$ is contained in a $(k-3)$-flat $\Lambda'$. Consider the sheaf of hyperplanes containing $\Lambda'$. They are $q+1$ and only $k-1$ of them contain other points of $\mM$ in addition to $P_1,\ldots, P_{k-1}$. Since $k-1<q+1$ there is at least one hyperplane $H$ such that $H\cap \mM=\{P_1,\ldots, P_{k-1}\}$ and $\langle H \cap \mM\rangle \subseteq \Lambda' \neq H$. This implies that, in this case, $\mM$ is not a cutting blocking set. Suppose then that $\dim(\Lambda)=k-2$. Fix $P_1,\ldots, P_{k-3}$ and consider the flat $\Gamma:=\langle P_1,\ldots, P_{k-3}, \ell_{k-2}\rangle$. If $\dim(\Gamma)<k-2$, then there exists $Q_{k-2}\in \ell_{k-2}\cap \langle P_1,\ldots, P_{k-3}\rangle$. Thus, if we replace $P_{k-2}$ by $Q_{k-2}$, we get that $\dim(\Lambda)<k-2$, and we can conclude as done before that $\mM$ is not cutting. Hence, assume $\dim(\Gamma)=k-2$. In this case $\Gamma \cap \ell_{k-1}\neq \emptyset$. Take $Q_{k-1} \in \Gamma \cap \ell_{k-1}$. If $Q_{k-1} \in \langle P_1,\ldots, P_{k-3}\rangle$, we substitute $P_{k-1}$ with $Q_{k-1}$ and get again $\dim(\Lambda)<k-2$, which implies $\mM$ not being cutting. Therefore, assume that the space $\langle P_1,\ldots,P_{k-3},Q_{k-1}\rangle$ is a hyperplane in $\Gamma$. Since $\Gamma$ also contains $\ell_{k-2}$, there exists $R_{k-2}\in \ell_{k-2}\cap \langle P_1,\ldots,P_{k-3},Q_{k-1}\rangle$. Thus, replacing $P_{k-1}$ with $Q_{k-1}$ and $P_{k-2}$ with $R_{k-2}$, we again obtain that $\dim(\Lambda)<k-2$ and~$\mM$ is not cutting.

 \noindent \underline{Case II:} $k=\sqrt{q}+2$ and $\mM$ is a $3$-dimensional Baer subspace. If $k\geq 5$, then $\langle \mM\rangle \neq \PG(k-1,q)$, 
 so~$\mM$ cannot be a cutting blocking set. For the remaining case, where $k=q=4$, one can observe that for a (hyper)plane $H$ in $\PG(3,q)$, $H$ intersects $\mM$ in a Baer subplane or in a Baer subline. In the latter case, one has $\langle \mM \cap H \rangle \neq H$, and so $\mM$ is not a cutting blocking set. The fact that for $k=q=4$ a $3$-dimensional Baer subspace $\mM$ cannot be cutting could be also deduced from Example~\ref{ex:[17,4]}, since the cardinality of $\mM$ is $15$.
 
 \noindent \underline{Case III:} $q=k=4$ and $\mM$ is the complement of a hyperoval in a plane of $\PG(3,q)$. In this case $\langle \mM \rangle \neq \PG(3,q)$ and $\mM$ cannot be a cutting blocking set.
 
 Therefore, when $4\leq k \leq \sqrt{q}+2$, the bound in Theorem~\ref{thm:lower_bound_length} is never sharp.
\end{remark}

\begin{corollary}\label{cor:lower_bound_not_sharp}
 Let $\mC$ be a minimal $[n,k]_q$ code with $3\leq k \leq \sqrt{q}+2$. Then $n \geq (q+1)(k-1)+1$, unless $q=2$ and $k=3$.
\end{corollary}

\begin{proof}
The case $k\geq 4$ has been discussed in Remark~\ref{rem:nonsharp_bound}.
When $k=3$, cutting blocking sets are equivalent to $2$-fold blocking set. Using Theorem~\ref{thm:2fold}, one can easily verify that the only case in which a $2$-fold blocking set has cardinality $2q+2$ is when $q=2$. 
\end{proof}

\bigskip

\section{Statistical Approach}
\label{sec:3}

Most bounds for minimal codes we are aware of involve 
either $(q,n,k)$, or $(q,k,d)$. Bounds involving all the four parameters $(q,n,k,d)$ can in turn be obtained combining these with  classical 
bounds for Hamming-metric codes, such as the Singleton or the Griesmer bound.

In this section, we develop a method to establish new inequalities that directly involve all the four parameters of a minimal code, namely $(q,n,k,d)$.
As an application, we obtain an upper bound for the minimum distance $d$ of a minimal code in terms of $(q,n,k)$. As we will see in the examples, this bound excludes the existence of minimal codes with parameter sets that do not violate any of the known bounds.

Our approach combines  Theorem~\ref{thm:lower_bound_old_conjecture} with ideas from statistics,  
interpreting the 
weight of the codewords of a linear code as a discrete random variable and computing/estimating its mean and variance.
As simple corollaries of our bounds, we recover classical results on constant-weight codes.

Throughout this section, $\mC$ denotes a nondegenerate code.
Our results can be made more precise when $\mC$ is projective;
see Section~\ref{sec:1} for the definition.

\subsection{Mean and Variance of the Nonzero Weights in a Linear Code}

We start with an upper bound for the sum of the squares of the weights in a nondegenerate linear code. The proof uses one of the  Pless' identities.

 \begin{lemma}
 \label{boundvar}
  Let $\mC$ be a nondegenerate $[n,k]_q$ code. We have
  $$\sum_{c \in \mC} \wH(c)^2 \ge q^{k-2} \, n \, (q-1) \, [n(q-1)+1].$$
  Moreover, equality holds if and only if $\mC$ is projective.
    \end{lemma}

  \begin{proof}
  For $i \in \{0,\ldots,n\}$ we denote by 
  $W_i(\mC^\perp)$ the number of codewords of weight $i$ in the dual code~$\mC^\perp$. Since $\mC$ is nondegenerate, we have 
  $W_1(\mC^\perp)=0$. Moreover, $\mC$ is projective if and only if $W_2(\mC^\perp)=0$.
    Using Pless' identities~\cite[Theorem 7.2.3(P1)]{huffman2010fundamentals} we can write
  \begin{equation} \label{pll}
  \sum_{c \in \mC} \wH(c)^2 =  
  \sum_{\nu =0}^2 \left( \nu! \, S(2,\nu) \, q^{k-\nu} (q-1)^{\nu} \binom{n}{n-\nu} \right) + 4 W_2(\mC^\perp) S(2,2)q^{k-2},
  \end{equation}
where $S(a,b) \ge 0$ is the Stirling number of the second kind indexed by $(a,b)$. 
Therefore
\begin{equation*}
  \sum_{c \in \mC} \wH(c)^2 \ge  
  \sum_{\nu =0}^2 \left( \nu! \, S(2,\nu) \, q^{k-\nu} (q-1)^{\nu} \binom{n}{n-\nu} \right),
  \end{equation*}
  with equality if and only if $\mC$ is projective.
  The lemma now follows from the fact that $S(2,0)=0$ and $S(2,1)=S(2,2)=1$.
\end{proof}

The next step consists in defining the mean and variance of the nonzero weights in a linear code and to study the latter via Lemma~\ref{boundvar}.
  
  \begin{notation} \label{EVar}
  For a code $\mC$, let
  \begin{align*}
  \E(\mC) &:= (q^k-1)^{-1} \sum_{c \in \mC} \wH(c), \\
  \Var(\mC) &:= (q^k-1)^{-1} \sum_{c \in \mC} \wH(c)^2 - \E(\mC)^2.
   \end{align*}
\end{notation}

We now compute/estimate these two quantities.

   \begin{theorem}
   \label{boundvar2}
    Let $\mC$ be a nondegenerate $[n,k]_q$ code. Let
    $\ell=n(q-1)/(q^k-1)$. We have
    $\E(\mC)=q^{k-1}\ell$ and 
    $\Var(\mC) \ge q^{k-2} \, \ell(1- \ell)$.
    Moreover, equality holds if and only if $\mC$ is projective.
   \end{theorem}
   
   \begin{proof}
Since $\mC$ is nondegenerate, we have
\begin{equation*}
\E(\mC) = (q^k-1)^{-1} \sum_{i=1}^n  (q^k-q^{k-1}) = n(q^k-q^{k-1})/(q^k-1) =  q^{k-1}\ell.
\end{equation*}
Combining this with
Lemma~\ref{boundvar} we obtain
\begin{eqnarray*}
\Var(\mC) &\ge&
\frac{q^{k-2}n(q-1)[n(q-1)+1]}{q^k-1} - q^{2k-2} \,  \frac{n^2(q-1)^2}{(q^k-1)^2} \\
&=& q^{k-2} \ell [n(q-1)+1] - \ell^2 q^{2k-2} \\
&=& q^{k-2} \ell(1-\ell),
\end{eqnarray*}
as desired.
   \end{proof}

   \begin{remark} \label{rrr}
  The quantity $\E(\mC)$ in Notation~\ref{EVar} expresses the \textbf{average weight}
   of $\mC$ and can be used to extend Theorem~\ref{firstn} as follows. Suppose that $\mC$ is a nondegenerate $[n,k]_q$ code with maximum weight~$w$.
   Using $\E(\mC) \le w$ one obtains
   $$n \ge \Bigg\lceil (n-w) \; \frac{q^k-1}{q^{k-1}-1} \Bigg\rceil \ge (n-w) q+1.$$
   In particular, if $\mC$ is minimal then
   $w\le n-k+1$ by~Proposition~\ref{prop:weights},
   from which Theorem~\ref{firstn} follows.
   \end{remark}

  As an immediate consequence of Theorem~\ref{boundvar2}
  we obtain the well-known fact  that constant-weight codes have large length; see e.g.~\cite{bonisoli1983every}.

\begin{corollary} \label{coro:const}
Let $\mC$ be a  constant-weight $[n,k,d]_q$ code. Then $n \ge (q^k-1)/(q-1)$. Moreover, if $\mC$ is projective then $n=(q^k-1)/(q-1)$ and $d=q^{k-1}$.
\end{corollary}

\begin{proof}
 Without loss of generality, $\mC$ is nondegenerate. By Theorem~\ref{boundvar2}
 we have $0 \ge q^{k-2}\ell(1-\ell)$,
 from which $\ell \ge 1$. If $\mC$ is projective then $\ell=1$ and so $d=\E(\mC)=q^{k-1}$, as claimed.
\end{proof}

\subsection{Bounds}

By applying the result of the previous subsection, we can finally derive an upper bound for the minimum distance of a code $\mC$ as a function of~$q$, $n$, $k$ and the maximum weight in $\mC$.  

\begin{theorem} \label{thm:ww}
Let $\mC$ be a nondegenerate $[n,k,d]_q$ code of maximum weight $w>d$.
Let $\ell=n(q-1)/(q^k-1)$. We have
$w > n(q^k-q^{k-1})/(q^k-1)$ and 
\begin{equation}\label{bd}
   d \le \Bigg\lfloor q^{k-1} \ell - \frac{q^{k-2} \ell(1-\ell)}{w-q^{k-1}\ell} \Bigg\rfloor.  
\end{equation}
Moreover, equality holds in~\eqref{bd} if and only if $\mC$ is a projective  two-weight code. 
\end{theorem}

\begin{proof}
The first inequality follows from the fact that $w > \E(\mC)$, since $\mC$ is not constant-weight. Using the inequality of Bhatia–Davis~\cite{bhatia2000better} we obtain
 $$\Var(\mC) \le (w-\E(\mC))(\E(\mC)-d).$$
 Since $\mC$ is nondegenerate and not constant-weight we have $\E(\mC)=q^{k-1}\ell<w$. Therefore we conclude by~Theorem~\ref{boundvar2}.

 The second part of the statement follows from the fact that the bound of Theorem~\ref{boundvar2} is sharp if and only if $\mC$ is projective, and that the Bhatia–Davis inequality is met with equality if and only if the 
 underlying distribution takes only two values; see~\cite[Proposition 1]{bhatia2000better}.
\end{proof}

As an application of Theorem~\ref{thm:ww} we obtain the following bound for the minimum distance of a minimal code.

\begin{corollary}
\label{coro:lbd}
\label{cor:constr}
Let $\mC$ be a minimal nondegenerate $[n,k,d]_q$ code. If $\mC$ is not constant-weight, then 
$n-k+1 > n(q^k-q^{k-1})/(q^k-1)$ and \begin{equation} \label{ll}
d \le \Bigg\lfloor \frac{n(q-1)q^{k-2}[n-1-q(k-1)]}{n(q^{k-1}-1)-(k-1)(q^k-1)}\Bigg\rfloor.
\end{equation}
In particular, we have
\begin{equation}\label{eq:lowerStat}
    q^{k-2}n^2 -Bn +C\geq 0,
\end{equation}
where
\begin{equation*}
\left\{ \; 
\begin{split}
    B&= q^{k-2} +(k-1)(2q^{k-1}-1) + \frac{q^{k-1}-1}{q-1}, \\
    C&= (k-1)^2(q^k-1)+\frac{(k-1)(q^{k}-1)}{q-1}.
\end{split} \right.
\end{equation*}
\end{corollary}

\begin{proof}
The maximum weight of $\mC$ satisfies $w \le n-k+1$ by Proposition~\ref{prop:weights}.
Combining this with Theorem~\ref{thm:ww} one gets
$n-k+1 > n(q^k-q^{k-1})/(q^k-1)$
and
\begin{equation} \label{lll}
d \le \Bigg\lfloor q^{k-1} \ell - \frac{q^{k-2} \ell(1-\ell)}{n-k+1-q^{k-1}\ell} \Bigg\rfloor,
\end{equation}where $\ell=n(q-1)/(q^k-1)$.
Lengthy computations show that the RHS 
of~\eqref{lll} is equal to the RHS of~\eqref{ll}. The second part of the statement follows by combining 
\eqref{ll} with Theorem~\ref{thm:lower_bound_old_conjecture}, after lengthy computations.
\end{proof}

\begin{example}\label{ex:[17,4]}
There is no minimal $[16,4]_4$ code. To see this, observe that if such a code existed, then \eqref{eq:lowerStat} would give
 $-42 \ge 0$, a contradiction. So the minimum length of a minimal code  of dimension $4$ over $\F_4$ is at least $17$. 
 
Consider the parameters $(q,n,k)=(4,17,4)$ and suppose that there exists an $[17,4]_4$ nondegenerate minimal code $\mC$. Since $n < (q^k-1)/(q-1)$, $\mC$ cannot be constant weight by Corollary~\ref{coro:const}. Therefore by Corollary~\ref{coro:lbd} we conclude that $d \le 10$. The existence of a minimal nondegenerate~$[17,4,11]_4$ code is therefore excluded by Corollary~\ref{coro:lbd}, but it is not excluded by any of the other known bounds for the parameters of minimal codes. Note moreover that, by Theorem~\ref{thm:lower_bound_old_conjecture}, we have $d\geq 10$. Therefore
the minimum distance of a putative 
$[17,4]_4$ nondegenerate minimal code is exactly $10$ (when the largest minimum distance of an ``unrestricted'' $[17,4]_4$ linear code is instead known to be $12$). 
\end{example}

\begin{remark}
The constraints imposed by Corollary~\ref{cor:constr} 
and Theorem~\ref{thm:lower_bound_length} are in general incomparable.
More precisely, each of the two results excludes the existence of some minimal codes that are not excluded by the other.

One can see that Corollary~\ref{cor:constr} improves on Theorem~\ref{thm:lower_bound_length} if and only if  \eqref{eq:lowerStat} is violated when specialized to $n=(q+1)(k-1)$. After lengthy computations, one sees that this happens if and only if 
$$k< \frac{q^{k-1}+2q^{k-2}-2q^{k-3}+q-2}{(q^{k-3}+1)(q-1)}.$$
Manipulating this inequality it can be checked that Corollary~\ref{cor:constr} improves on Theorem~\ref{thm:lower_bound_length} for the parameter set $\{(k,q) \mid 3\leq k \leq q+3, \, q\geq 3\}$. When $q$ is at least $3$, this is an improvement also on Corollary~\ref{cor:lower_bound_not_sharp}.
On the other hand, Theorem~\ref{thm:lower_bound_length} provides a strictly sharper estimate than Corollary~\ref{cor:constr} if and only if  \eqref{eq:lowerStat} is satisfied for $n=(q+1)(k-1)-1$. For instance, this happens for the parameter set
$\{(k,q):k \geq 2q\}$.

We include in Table~\ref{indep} three collections of parameter sets that are excluded by Theorem~\ref{thm:lower_bound_length}
and Corollary~\ref{cor:constr}. The first column contains parameters that are excluded by both results, while the other two contain parameters that are  excluded 
by either Theorem~\ref{thm:lower_bound_length} or Corollary~\ref{cor:constr} (and not by both).
{\small
\begin{table}[h!]
    \centering
    \begin{tabular}{|x{4.5cm}|x{4.5cm}|x{4.5cm}|}
    \hline
    Some parameters of minimal codes  excluded by both Theorem~\ref{thm:lower_bound_length} and  Corollary~\ref{cor:constr}    &
     Some parameters of minimal codes excluded by Theorem~\ref{thm:lower_bound_length} and not by  Corollary~\ref{cor:constr}   &
     Some parameters of minimal codes
     excluded by Corollary~\ref{cor:constr} 
     and not by 
     Theorem~\ref{thm:lower_bound_length}
     \\
     \hline
     \hline
   $[8,4]_2$  & $[17,7]_2$     & $[16,5]_3$ \\ 
   $[15,5]_3$  &  $[31,9]_3$     & $[16,4]_4$ \\
   $[24,6]_4$ & $[44,10]_4$     & $[25,6]_4$ \\
    $[35,7]_5$ & $[99,21]_4$   &  $[36,7]_5$  \\
   $[63,9]_7$ & $[65,12]_5$  &   $[26,4]_7$ \\
    \hline
    \end{tabular}
    \caption{Code parameters for which the existence of minimal codes is excluded by Theorem~\ref{thm:lower_bound_length} and/or Corollary~\ref{cor:constr}.}
    \label{indep}
\end{table}
}
\end{remark}

\subsection{Other Applications}

In this short subsection we illustrate how  Theorem~\ref{thm:ww}
can be applied to study codes that are not necessarily minimal.
We start with 
a generalization of Corollary~\ref{coro:const}. More precisely, we show that the relative difference between the maximum and minimum weight of a code, $(w-d)/n$, gives a lower bound on the code's length. In other words, if the maximum and minimum weight of a code are relatively close to each other, then the code length is necessarily large.

\begin{proposition} \label{thm:ngrows}
 Let $\mC$ be a nondegenerate $[n,k,d]_q$ code of maximum weight $w$. We have
 $$\frac{1}{n} \le \frac{q-1}{q^k-1} + \frac{1}{4} \left(\frac{w-d}{n}  \right)^2 \frac{q^k-1}{q^{k-2}(q-1)}.$$
\end{proposition}

Note that in the extreme case where $w=d$ we recover 
Corollary~\ref{coro:const}.

\begin{proof}[Proof of Proposition~\ref{thm:ngrows}]
Using Popoviciu's inequality for the variance, along with Theorem~\ref{boundvar2}, we find
 $$q^{k-2} \ell(1-\ell) \le \Var(\mC) \le 
 \frac{1}{4} (w-d)^2,$$
 where $\ell=n(q-1)/(q^k-1)$.
 Re-arranging the terms, after tedious computations one obtains the desired inequality.
\end{proof}

A second application of 
Theorem~\ref{thm:ww} consists in obtaining constraints on the parameters of a code having few weights. A classical result 
about these codes is the following theorem by Delsarte.

\begin{theorem}[see~\cite{delsarte1973four}] \label{dee}
Let $\mC$ be an $[n,k]_q$ code and let $s=|\{\omega(c) \mid c \in \mC, \, c \neq 0\}|$. We have
$$q^k \le \sum_{i=0}^s \binom{n}{i} (q-1)^i.$$
\end{theorem}

Specializing to $s=2$, the previous theorem shows that, for example, any
two-weight $[n,k]_q$ code satisfies
\begin{equation} \label{aaa}
q^k \le 1 + n(q-1) + \frac{n(n-1)}{2} (q-1)^2.
\end{equation}
This result however does not take into account \textit{which} values
the weight distribution can take.
Exploiting this information,
Theorem~\ref{thm:ww} provides in general 
different constraints 
on $n$ than those in~\eqref{aaa}.
We illustrate this with an example.

\begin{example}
Following the notation of 
 Theorem~\ref{thm:ww} and Theorem~\ref{dee}, let
 $(q,k,s,d,w)=(2,8,2,16,24)$. We look for a nondegenerate binary two-weight code $\mC$ of dimension $8$ having weights $16$ and $24$. 
 The constraints imposed on $n$ by Theorem~\ref{thm:ww} imply $34 \le n \le 45$, where the upper bound is met with equality if $\mC$ is projective. The constraint imposed by~\eqref{aaa} is instead $n \ge 23$. It is known that there exists a projective binary two-weight code of parameters $(n,k,d,w)=(45,8,16,24)$.
\end{example}

\bigskip

\section{Geometric Approach}
\label{sec:4}

 As already illustrated in Subsection~\ref{sec:mincodescut}, minimal codes are in one-to-one correspondence with cutting blocking sets. In this section we focus on this point of view on minimal codes, exploiting their 
 geometric characterization to construct new, general and infinite families of minimal codes. 
 In particular, we provide a construction of cutting blocking sets derived from Desarguesian $(r-1)$-spreads of $\PG(rt-1,q)$. In turn, this leads to a an inductive construction of small cutting blocking sets or, equivalently, of minimal codes with short length.
 In contrast to previous approaches, our construction works over any (possibly very small) finite field.

\subsection{Minimal Codes from Spreads}
We start by recalling the definition of $t$-spread in $\PG(k-1,q)$, which we will use to obtain a new construction of minimal codes. 
A $t$-\textbf{spread} $S$ of $\PG(k-1, q)$ is a partition of $\PG(k-1,q)$ in $t$-flats. It is well known that such a $t$-spread exists if and only if $t+1$ divides $k$; see~\cite{segre1964teoria}. In particular, a $1$-spread of  $\PG(k-1, q)$ is a partition of its points into disjoint lines and it is also called a \textbf{linespread}. It exists if and only if $k$ is even.

An algebraic representation of an $(r-1)$-spread of $\PG(2r-1, q)$ can be obtained as follows.
Let $\gamma\in \F_{q^r}$ be a primitive element and let $M\in \F_q^{r\times r}$ be the companion matrix of the minimal polynomial of $\gamma$ over $\Fq$. It is well known that $\F_{q^r}\cong\Fq[\gamma]\cong \Fq[M]=\{0\}\cup\{M^i : 1\leq i \leq q^r-1\}$ as $\F_q$-algebras.   For $i\in[q^r-1]$ define $V_i:=\{[x:xM^i]\mid x\in \PG(r-1,q)\}$, $V_0:=\{[x:0]\mid x\in \PG(r-1,q)\}$ and $V_{q^r}:=\{[0:y]\mid y\in \PG(r-1,q)\}$.
Then the set $\{V_0,\dots, V_{q^r}\}$ is an $(r-1)$-spread of $\PG(2r-1, q)$.

\begin{theorem}\label{thm:ConstructionPlanarSpreads}
Let $S$ be the $(r-1)$-spread of $\PG(2r-1,q)$ defined above and let $\mathcal{B}=V_0 \cup V_i \cup V_j \cup V_{q^r}\subseteq \PG(2r-1,q)$, with $0<i<j<q^r$. Suppose that for every $s>1$ dividing $r$ we have $j-i \not\equiv 0 \mod \big(\frac{q^s-1}{q-1}\big)$. Then $\mathcal{B} $ is a cutting blocking set.
\end{theorem}

\begin{proof}
 For ease of exposition we switch to vector notation, in which we represent $V_0$, $V_i$, $V_j$, $V_{q^2}$ as elements of the Grassmannian $\mathrm{Gr}_q(r,2r)$. In this representation we have $V_0=\mathrm{rowsp}(I_r \mid 0)$,
 $V_{q^r}=\mathrm{rowsp}(0 \mid I_r)$, $V_i=\mathrm{rowsp}(I_r \mid M^i)$ and $V_j=\mathrm{rowsp}(I_r \mid M^j)$. Let $H$ be a hyperplane in~$\Fq^{2r}$. We want to show that $\langle H \cap \mB\rangle =H$, or, equivalently, that $\langle H \cap \mB\rangle = \langle H \cap V_0 \rangle + \langle H \cap V_{q^r} \rangle+\langle H \cap V_i \rangle+\langle H \cap V_j \rangle$ has dimension at least $2r-1$. Observe first that if $H$ contains one among the $V_\ell$'s, say $V_0$, then there is nothing to prove, since  $\langle H \cap V_0\rangle+\langle H \cap V_i\rangle$ has already dimension (at least) $2r-1$. 
 Hence we can assume that $H$ intersect both $V_0$ and $V_{q^r}$ in an $(r-1)$-dimensional subspace. Then the space $\langle H \cap V_0\rangle+\langle H \cap V_{q^r}\rangle$ has dimension $2r-2$.  We can write  the intersection spaces as 
 $$\begin{array}{lclccclllcl}
 H \cap V_0 & = & \mathrm{rowsp}(&\!\!\!\!X_1&\!\!\!\! \mid &\!\!\!\!0&\!\!\!\!), & \phantom{aaaaa} & H \cap V_i & = & \mathrm{rowsp}(X_2 \mid X_2M^i), \\
  H \cap V_{q^r} & = & \mathrm{rowsp}(&\!\!\!\!0 &\!\!\!\!\mid &\!\!\!\!X_3&\!\!\!\!), & \phantom{aaaaa}
&  H \cap V_j & = & \mathrm{rowsp}(X_4 \mid X_4M^j),
 \end{array}$$
 for some $X_1,X_2,X_3,X_4 \in \Fq^{(r-1)\times r}$ of rank $r-1$.
 
Suppose by contradiction that $\langle H \cap \mB\rangle$ has dimension exactly $2r-2$. This implies 
 $$ \mathrm{rowsp}\begin{pmatrix} X_2 & X_2M^i \\X_4 & X_4M^j
 \end{pmatrix} \subseteq \mathrm{rowsp}\begin{pmatrix} X_1 & 0 \\ 0 & X_3 \end{pmatrix}, $$
 which in turn implies  
 $\mathrm{rowsp}(X_2)=\mathrm{rowsp}(X_4)=\mathrm{rowsp}(X_1)$ and $\mathrm{rowsp}(X_3)=\mathrm{rowsp}(X_2M^i)=\mathrm{rowsp}(X_4M^j)$. Without loss of generality, we can assume that $X_1=X_2=X_4=:X$, which reduces the above condition to 
 $$\mathrm{rowsp}(X_3)=\mathrm{rowsp}(XM^i)=\mathrm{rowsp}(XM^j).$$
 Thus, there exists a matrix $A \in \GL(r-1,q)$ such that 
 \begin{equation}\label{eq:sylvester}
 AX-XM^{j-i}=0.
 \end{equation}
The matrix equation in \eqref{eq:sylvester}, where the matrix $X$ is the unknown, is a \emph{Sylvester equation}. This 
is known to 
 have a unique solution 
 if the minimal polynomials of $A$ and $M^{j-i}$ are coprime; see e.g. \cite[Theroem 2.4.4.1]{horn2013matrix}. Observe that the minimal polynomial of $M^{j-i}$ is irreducible of degree~$r$, since $M^{j-i}$ corresponds to the element $\gamma^{j-i}$ and by the assumption on~$j-i$ in the statement we have $\Fq[\gamma^{j-i}]=\F_{q^r}$. Moreover, the minimal polynomial of $A$ has degree at most $r-1$, and hence it is coprime with the one of $M^{j-i}$'s. Therefore \eqref{eq:sylvester} has a unique solution, which is clearly $X=0$. This leads to a contradiction and concludes the proof.  
\end{proof}

We now concentrate on the more general case of 
 $(r-1)$-spreads in $\PG(rt-1,q)$.
 These can be constructed
 using the so-called \textbf{field reduction}; see \cite{segre1964teoria, lavrauw2015field}. This technique identifies points in $\PG(t-1,q^r)$ with $(r-1)$-flats in $\PG(rt-1,q)$. The idea is exactly the same as for the algebraic $(r-1)$-spread of $\PG(2r-1,q)$ described above. Let $\gamma$ be a primitive element in~$\F_{q^r}$ and let $M$ be the companion matrix of the minimal polynomial of $\gamma$ over $\Fq$. As already explained, there is an isomorphism $\F_{q^r} \cong \Fq[M]=\{0\} \cup \{M^i: 1\leq i \leq q^r-1\}$, which we call $\phi$. We can then extend it to vectors in $\F_{q^r}^t$ componentwise, obtaining an injective map
$$\begin{array}{rcl}
\varphi :  \F_{q^r}^t & \longrightarrow & \Fq^{r \times rt} \\
(v_1,\ldots, v_t) & \longmapsto &(\phi(v_1) \mid \ldots \mid \phi(v_t)).
\end{array}$$
This map can in turn be extended to a map  $\bar{\varphi}:\PG(t-1,q^r) \longrightarrow \mathrm{Gr}_q(r,tr)$, the Grassmannian, defined by $P=[v]\longmapsto \mathrm{rowsp}(\varphi(v))$. Note that $\bar{\varphi}$ is well-defined since it does not depend on the choice of the representative $v$ for the point $P$. Indeed, for a nonzero scalar multiple of $v$, say $\gamma^iv$, we have $\varphi(\gamma^iv)=M^i\varphi(v)$ and since $M^i$ is invertible, $\mathrm{rowsp}(M^i\varphi(v))=\mathrm{rowsp}(\varphi(v))$. It is then well-known that $\mathrm{Im}(\bar{\varphi})$ is a (vectorial) $r$-spread of $\Fq^{rt}$, which naturally gives rise to a projective  $(r-1)$-spread of $\PG(rt-1,q)$. Such a spread is known as \textbf{Desarguesian spread}; see \cite{segre1964teoria}.

In the sequel we will need the following special points in $\PG(t-1,q^r)$:  $P_\ell:=[e_\ell]$ for $\ell\in[t]$ and $Q_{\ell,m,i}:=[u_{\ell,m,i}]$, where $u_{\ell,m,i}:=e_\ell+\gamma^ie_{m}$ for $1\leq\ell<m\leq t$ and  $i \in [q^r-1]$. These will be used in the next result to extend the construction of Theorem \ref{thm:ConstructionPlanarSpreads} from two to $t$ \textbf{blocks}.

\begin{theorem}\label{thm:ConstructionGeneralSpreads}
  For each pair of integers $(\ell, m)$ such that $1\leq \ell <m \leq t$, let  $j_{\ell,m}, i_{\ell,m}\in [q^r-1]$ be integers with the following property: for all $s>1$ dividing $r$, $j_{\ell,m} -i_{\ell,m}\not\equiv 0 \mod \big(\frac{q^s-1}{q-1}\big)$. Define the set
 $$\mT:=\bigg(\bigcup_{1\leq \ell \leq t}\bar\varphi(P_\ell)\bigg) \cup \bigg(\bigcup_{1\leq \ell <m \leq t} (\bar\varphi(Q_{\ell,m,i_{\ell,m}}) \cup \bar\varphi(Q_{\ell,m,j_{\ell,m}}))\bigg).$$
 Then the projectivization of $\mT$ is a cutting blocking set in $\PG(rt-1,q)$.
\end{theorem}

\begin{proof}
Once again we work in vector notation.
Let $H$ be a hyperplane in $\Fq^{rt}$. Let $a:=|\{\ell: \bar\varphi(P_\ell) \subseteq H \}|$. Then $0\leq a\leq t-1$ and $\dim(\langle H \cap \mT\rangle)\geq ra+(r-1)(t-a)$. Without loss of generality assume that $\{\ell \, : \,  \bar\varphi(P_\ell) \subseteq H \}=[a]$. Hence $\langle H\cap \mT\rangle$ contains the span of the first~$a\cdot r$ standard basis vectors.  By taking the quotient on this span, we reduce ourselves to proving the same statement for $a=0$, replacing $t$ by $t-a$. Therefore we can also assume $a=0$ without loss of generality.

We have that $\Lambda:=\langle H\cap \big(\bigcup_\ell \bar\varphi(P_\ell)\big)\rangle$ has dimension $(r-1)t$. For all integers $1\leq \ell <m \leq t$, define \begin{align*}
    \mS_{\ell,m}&:=\bar\varphi(P_\ell)\cup \bar\varphi(P_m)\cup  \bar\varphi(Q_{\ell,m,i_{\ell,m}})\cup \bar\varphi(Q_{\ell,m,j_{\ell,m}}), \\
\Pi_{\ell,m}&:=\langle \bar\varphi(P_\ell) \cup \bar\varphi(P_m)  \rangle=\langle e_i :  (\ell-1)r+1 \leq i \leq \ell r, \mbox{ or } (m-1)r\leq i \leq mr\rangle.
\end{align*}
Then $H\cap \Pi_{\ell,m}$ is a hyperplane in $\Pi_{\ell,m} \cong \Fq^{2r}$. Moreover, using the same argument as in the proof of Theorem~\ref{thm:ConstructionPlanarSpreads},
there exists a vector $v_{\ell,m}\in (H\cap \Pi_{\ell,m}) \cap \mS_{\ell,m}\subseteq H\cap \mS_{\ell,m}$ such that $v_{\ell,m} \notin \langle H \cap (\bar\varphi(P_\ell) \cup \bar\varphi(P_m))\rangle$. Observe that the support of $v_{\ell,m}$ is contained only in the $\ell$-th and the $m$-th blocks and that we can write $v_{\ell,m}=w_{\ell,m}^{(\ell)}+w_{\ell,m}^{(m)}$, where $ w_{\ell,m}^{(j)}\in \langle \bar\varphi(P_j)\rangle$, i.e., it has support contained only in the $j$-th block, for $j \in \{\ell,m\}$. Now consider the $t-1$ vectors $v_{1,2},\ldots, v_{1,t}$. Since $a=0$, none of the $v_{1,i}$'s belongs to $\Lambda$. It is left to show that for each $i\geq 3$ we have $v_{1,i}\notin \Gamma_{i-1}:=\Lambda + \langle v_{1,2}, \ldots, v_{1,i-1}\rangle$. By contradiction, suppose that $v_{1,i} \in \Gamma_{i-1}$. Let $\rho_i:\Fq^{rt}\rightarrow\Fq^r$ denote the projection on the $i$-th block. We have
$$w_{1,i}^{(i)}=\rho_i(v_{1,i}) \in \rho_i(\Gamma_i)=\langle H\cap \bar\varphi(P_i)\rangle,$$
since, by construction, the $i$-th block of any vector in $\Gamma_{i-1}$ is equal to the $i$-th block of some element in $\bar\varphi(P_i)\cap H$. Therefore, also the vector $w_{1,i}^{(1)}=v_{1,i}-w_{1,i}^{(i)}$ belongs to $H$. This means that $v_{1,i}\in H\cap (\bar\varphi(P_1) \cup \bar\varphi(P_i))\subseteq \Lambda$, which leads to a contradiction.
\end{proof}

\begin{remark}\label{rem:generalize_davydov}
The construction of Theorem \ref{thm:ConstructionGeneralSpreads} for $r=t=2$ (or, equivalently, the one of Theorem \ref{thm:ConstructionPlanarSpreads} for $r=2$) coincides with the construction of cutting blocking sets of \cite[Theorem~3.7]{davydov2011linear}, which consists of $4$ disjoint lines in $\PG(3,q)$. Therefore, Theorem \ref{thm:ConstructionGeneralSpreads} can be viewed as a generalization of that result.
\end{remark}

\begin{example}\label{exa:cutting_from_spread}
 We explicitly construct a cutting blocking set in $\PG(5,q)$ as explained in Theorem~\ref{thm:ConstructionGeneralSpreads}, with $r=2$ and $t=3$. We take as $\gamma$ a primitive element of $\F_{q^2}$ whose minimal polynomial over $\Fq$ is $x^2-p_1x-p_0$. We have $P_1=[1:0:0]$, $P_2=[0:1:0]$, $P_3=[0:0:1]$ and choose the following points in $\PG(2,q^2)$: $Q_{1,2,q^2-1}=[1:1:0]$, $Q_{1,2,1}=[1:\gamma:0]$, $Q_{1,3,q^2-1}=[1:0:1]$, $Q_{1,3,1}=[1:0:\gamma]$, $Q_{2,3,q^2-1}=[0:1:1]$, $Q_{2,3,1}=[0:1:\gamma]$.
 Therefore the set $\mT$ is
 \begin{align*}\mT= &\;\{(x,y,0,0,0,0): x,y \in \Fq\}\cup \{(0,0,x,y,0,0): x,y \in \Fq\}\cup\{(0,0,0,0,x,y) : x,y\in \Fq\}\\
 &\cup \{(x,y,x,y,0,0) : x,y \in \Fq \}\cup \{(x,y,y,p_0x+p_1y,0,0): x,y \in \Fq \} \\ 
 &\cup \{(x,y,0,0,x,y) : x,y \in \Fq \}\cup \{(x,y,0,0,y,p_0x+p_1y): x,y \in \Fq \} \\
 & \cup \{(0,0,x,y,x,y) : x,y \in \Fq \}\cup \{(0,0,x,y,y,p_0x+p_1y): x,y \in \Fq \}.\end{align*}
 The projectivization of $\mT$ gives the desired cutting blocking set in $\PG(5,q)$.
 
\end{example}

\subsection{Inductive Constructions of Cutting Blocking Sets}
As already observed in  the Introduction, of particular interest is the study minimal codes of small length for a given dimension.  Formally, for a fixed positive integer $k$ and a prime power~$q$, we are interested in determining the value of
$$m(k,q):=\min\left\{n \in \mathbb N_{\ge 1} \mid  \mbox{ there exists a minimal } [n,k]_q \mbox{ code}  \right\}.$$
This function has been explicitly studied in \cite{lu2019parameters}, where it was observed  that $m(2,q)=q+1$ and that 
\begin{equation}\label{eq:bounds_mkq}q(k-1)+1 \leq m(k,q) \leq (q-1)\binom{k}{2}+k,\end{equation}
where the upper bound is \emph{constructive} (the tetrahedron from page~\pageref{tetra}). The same results were independently obtained in \cite{alfarano2019geometric}, where shorter minimal codes are constructed for $k \in \{3,4,5\}$. In this notation, Theorem~\ref{thm:lower_bound_length} improves on the lower bound in \eqref{eq:bounds_mkq}, reading
$$ m(k,q) \geq (q+1)(k-1).$$
We already obtained improvements on this bound in Corollary~\ref{cor:lower_bound_not_sharp} and  Corollary~\ref{cor:constr}, as shown in Table~\ref{indep}. 

 In \cite{chabanne2013towards} it has been shown that  the upper bound on $m(k,q)$ in \eqref{eq:bounds_mkq} is far from being tight. More precisely, one has
\begin{equation}\label{eq:upper_bound_nonconstructive} m(k,q) \leq \frac{2k}{\log_q\left(\frac{q^2}{q^2-q+1}\right)},\end{equation}
indicating that, in principle, for a fixed $q$ and $k$ large enough
one might construct much shorter minimal codes. In particular, a natural problem is 
that of finding, for a fixed $q$,
 an infinite family of minimal codes over $\Fq$ whose length is {linear} in $k$. 
This problem is naturally motivated by the goal of
\textit{explicitly} constructing asymptotically good minimal codes. Indeed, while these codes are known to be asymptotically good, the proofs in~\cite{MR3163591,alfarano2019geometric}
are not constructive, as well as the bound in \eqref{eq:upper_bound_nonconstructive}. 
We are currently unaware of any \emph{explicit} general construction of minimal codes whose length is unbounded for a \textit{fixed} $q$, and that are asymptotically shorter than the tetrahedron; see also the discussion in Remark~\ref{remdisc}.

In the sequel, we introduce two new families of minimal codes whose  lengths are shorter than the one of the tetrahedron by a factor $2$ and by a factor $\frac{9}{4}$, respectively. We start with a result that represents a first step towards inductive constructions of cutting blocking sets.

\medskip

\begin{proposition}\label{prop:inductiveConstruction}
 Let $\mB=\mB_1\cup\ldots \cup \mB_r$ be a cutting blocking set in $\PG(N,q)$. For each $i \in [r]$, let $\Gamma_i:=\langle \mB_i\rangle \cong \PG(n_i,q)$ for some $n_i \leq N$ and let $\mB_i'\subseteq \Gamma_i$ be the isomorphic image of a cutting blocking set in $\PG(n_i,q)$.
Then  $\mB':=\mB_1'\cup\ldots\cup\mB_r'$ is a cutting blocking set.
\end{proposition}

\begin{proof}
 Let $H$ be a hyperplane in $\PG(N,q)$. We want to show that $\langle H \cap \mB'\rangle=H$. By hypothesis we have that 
 $$H=\langle H\cap \mB\rangle=\langle H \cap \mB_1\rangle + \ldots + \langle H \cap \mB_r\rangle.$$
 Consider the spaces $\Lambda_i:=H \cap \langle \mB_i\rangle$, $i \in [r]$. Clearly, $\Lambda_i\supseteq \langle H \cap \mB_i\rangle$ for all $i$. We now examine two cases separately.
 
 \noindent\underline{Case I:} $\Lambda_i=\langle \mB_i\rangle$, that is, $H$ contains $\langle \mB_i\rangle$. In this case $H$ also contains $\mB_i'$ and $\langle H\cap \mB_i'\rangle =\langle \mB_i'\rangle = \langle \mB_i \rangle=\Lambda_i$.
 
 \noindent\underline{Case II:} $\Lambda_i$ is a hyperplane  in $\langle \mB_i\rangle$. By hypothesis, $\mB_i'$ is a cutting blocking set in $\langle \mB_i\rangle$, and hence $\langle H \cap \mB_i'\rangle \supseteq \langle\Lambda_i \cap \mB_i'\rangle=\Lambda_i$. 
 
 Therefore in both cases we have
 \begin{align*}\langle H \cap \mB'\rangle &= \langle H \cap \mB_1'\rangle + \ldots + \langle H \cap \mB_r'\rangle  \supseteq \Lambda_1+\ldots+\Lambda_r \\ &\supseteq \langle H \cap \mB_1\rangle + \ldots + \langle H \cap \mB_r\rangle=H,\end{align*}
 concluding the proof.
\end{proof}

We are now ready to combine the above result with Theorem~\ref{thm:ConstructionGeneralSpreads}
and derive a recursive upper bound on $m(k,q)$.

\begin{theorem}\label{thm:inductiveCardinality}
For all positive $a,b \in \mathbb{N}$,
 $$m(ab,q) \leq a^2m(b,q).$$
\end{theorem}

\begin{proof}
 By Theorem~\ref{thm:ConstructionGeneralSpreads} we know that we can construct a cutting blocking set in $\PG(ab-1,q)$ with the aid of a $(b-1)$-spread. More precisely, we only need to take $a^2$ disjoint $(b-1)$-flats $\Gamma_1,\ldots, \Gamma_{a^2} \cong \PG(b-1,q)$ from the spread. By Proposition~\ref{prop:inductiveConstruction}, for each of them we can take the isomorphic image of a cutting blocking set in $\PG(b-1,q)$ with minimum cardinality $m(b,q)$. Therefore, we finally obtain a cutting blocking set in $\PG(ab-1,q)$ of cardinality $a^2m(b,q)$.
\end{proof}

Observe that the proof of Theorem~\ref{thm:inductiveCardinality} gives an explicit way of constructing a minimal $[a^2m(b,q),ab]_q$ code, provided that there exists already a construction for an $[m(b,q),b]_q$ minimal code. We illustrate how this construction works 
with the following example.

\begin{example} We fix 
 $k=6=3 \cdot 2$ and assume $q$ to be a square.  Observe that under these assumptions we know the exact values of $m(2,q)$ and  $m(3,q)$; see Section~\ref{sec:mincodescut}. Namely, we have $m(2,q)=q+1$ and
 $$m(3,q)=\begin{cases} 3q & \mbox{ if } q=4, \\
 2(q+\sqrt{q}+1) & \mbox{ if } q \geq 9.\end{cases}$$ 
 Now we can use  Theorem~\ref{thm:inductiveCardinality} in two ways. On the one hand, we deduce that
 $$ m(6,q) \leq 9 \cdot m(2,q)=9(q+1).$$
 Such a construction is obtained by  taking $9$ lines from a linespread in $\PG(5,q)$ as explained also in Example~\ref{exa:cutting_from_spread}.
 On the other hand, by interchanging the roles of $2$ and $3$ we obtain
 $$m(6,q) \leq 4\cdot m(3,q)= \begin{cases} 12(q+1) & \mbox{ if } q=4, \\
 8(q+\sqrt{q}+1) & \mbox{ if } q \geq 9.
 \end{cases}$$
 The corresponding cutting blocking set is constructed by first selecting $4$ planes in $\PG(5,q)$ via Theorem~\ref{thm:ConstructionPlanarSpreads}, and then by choosing, in each of these planes, a minimal $2$-fold blocking set: when $q=4$, we take $3$ lines not intersecting all in the same point; when $q\geq 9$, we choose $2$ disjoint Baer subplanes.
 It is easy to check that for $q<64$ the cutting blocking set consisting of $9$ lines is smaller, while for $q\geq 64$ the $8$ Baer subplanes give rise to a cutting blocking set with smaller cardinality. Notice that both  constructions produce a smaller cutting blocking set than the tetrahedron, which contains $15q-9$ points. For instance, let us consider the case $q=4$. The $9$ lines give rise to a minimal $[45,6]_4$ code, the $8$ Baer subplanes lead to a minimal $[56,6]_4$ code, while the tetrahedron  provides a $[66,4]_4$ code.
 If we take $q=64$, then the three constructions produce minimal codes whose parameters are $[585,6]_{64}$, $[584,6]_{64}$ and $[966,6]_{64}$, respectively.  
\end{example}

\begin{remark}
Very recently, a construction of cutting blocking sets in $\PG(5,q)$  as union of seven disjoint lines has been 
given in \cite{bartoli2020cutting}. This gives 
an improvement on the known upper bound for $m(6,q)$. In the same work, a construction of a cutting blocking set in $\PG(3,q^3)$ of size $3(q^3+q^2+q+1)$ has been obtained as union of three suitable disjoint $q$-order subgeometries. These results together yield the following bounds:
 \begin{align*}
     m(4,q^3) & \leq 3(q^3+q^2+q+1), \\
     m(6,q) & \leq 7(q+1).
 \end{align*}
\end{remark}

The proof of Theorem~\ref{thm:inductiveCardinality}, which constructs
minimal codes of dimension~$k=ab$, heavily relies on the 
existence of a smaller minimal code, whose dimension divides~$k$. 
Clearly, this recursive construction does not cover all dimensions,
as for instance it does not provide any nontrivial minimal code of prime dimension. While for~$k=5$ one can rely on the construction provided in \cite[Construction 2]{alfarano2019geometric}, which gives a $[8q-3,5]_q$ minimal code, for primes greater than~$5$ we are not (yet) able to construct any \emph{short} minimal code different from the tetrahedron. Also, we are not (yet) able to construct {short} minimal codes of odd dimension, unless the latter is divisible by~$3$ and~$q$ is a square. When $k$ is odd one can construct minimal codes taking several $(r-1)$-flats in $\PG(k-1,q)$, where $r$ is the smallest prime dividing $k$. However, when such a prime is big, the resulting code turns out to be quite long. 

The discussion in the previous paragraph motivates us 
to look for alternative constructions of minimal codes, with the ultimate goal of covering a larger dimension range.
Our next move in this direction is an 
inductive result that allows us to construct a cutting blocking set in $\PG(k,q)$ starting from a smaller one in $\PG(k-1,q)$. The following result has already been shown in \cite[Construction A]{davydov2011linear}. We include a proof for completeness.

\begin{proposition}[\textnormal{see \cite[Theorem 3.10]{davydov2011linear}}]\label{prop:inductivek+1}
Let $\mB'$ be a cutting blocking set in $\PG(k-1,q)$. Fix a hyperplane $\Lambda\subseteq \PG(k,q)$ and take an isomorphic image $\mT$ of $\mB'$ in $\Lambda$. Moreover, select $k$ points $P_1,\ldots, P_k \in \langle\mT\rangle$ not lying all in the same $(k-2)$-flat and a point $P \in \PG(k,q)\setminus \Lambda$. Define the lines $\ell_i:=\langle P_i, P \rangle$. Then the set 
$$\mB:= \mT \cup \bigg(\bigcup_{i=1}^k \ell_i\setminus\{P_i\}\bigg)$$
is a cutting blocking set in $\PG(k,q)$.  In particular, for every $k\in \N_{\ge 1}$ we have
$$m(k+1,q) \leq m(k,q)+(q-1)k+1.$$
\end{proposition}

\begin{proof}
 Let $H$ be a hyperplane in $\PG(k,q)$. If $H=\Lambda$, then clearly $\langle H \cap \mB\rangle=H$. If $H\neq \Lambda$, then we have that $\Lambda_0:=H\cap \Lambda$ is a hyperplane in $\Lambda$. Hence $\langle H\cap \mT\rangle= \langle \Lambda_0\cap \mT\rangle=\Lambda_0$. Moreover, $H$ meets each of the lines $\ell_i$'s in a point $Q_i$. Observe that not all of them can lie in $\Lambda$, because otherwise we would have $Q_i=P_i$ for every $i$ and $H=\Lambda$. Therefore, there exists a point $Q_i \in (\ell_i\cap H) \setminus \langle\mT\rangle$. This implies that $Q_i \in \langle H\cap\mB \rangle\setminus \Lambda_0$ and we can conclude that 
 $\langle H \cap \mB\rangle =H$.
\end{proof}

Proposition~\ref{prop:inductivek+1} shows how to construct a cutting blocking set in $\PG(k,q)$ which contains a copy of a cutting blocking set $\mT$ in $\PG(k-1,q)$. This is achieved by adding $(q-1)k+1$ points to~$\mT$. Moreover, among cutting blocking sets containing a copy of a smaller cutting blocking set (of codimension $1$), the construction of Proposition~\ref{prop:inductivek+1} is optimal, as shown by the following result.

\begin{proposition} \label{prop:not}
Let $\mB\subseteq \PG(k,q)$ be a cutting blocking set such that it contains (an isomorphic image of) a cutting blocking set $\mB'$ of $\PG(k-1,q)$. Then 
$$ |\mB|\geq |\mB'|+ (q-1)k+1.$$
\end{proposition}

\begin{proof}
 Let $\mB$ be a cutting blocking set in $\PG(k,q)$ and suppose it contains a copy $\mB'$ of a cutting blocking set in $\PG(k-1,q)$. Then $\mB'$ is contained in a 
 hyperplane $H$. By the correspondence between linear codes and projective systems (see page \pageref{page:correspondence}) we have
 $$d \leq |\mB|-|\mB\cap H|\leq |\mB|-|\mB'|.$$
 Combining this  with Theorem~\ref{thm:lower_bound_old_conjecture} we obtain the desired inequality.
\end{proof}

\begin{remark}
Proposition~\ref{prop:not} shows that the inductive construction  from Proposition~\ref{prop:inductivek+1} gives rise to a cutting blocking set that is minimal among all the cutting blocking sets containing a given cutting blocking set of codimension $1$. It is interesting to observe that starting from $\PG(1,q)$ and iterating this construction $k$ times, one obtains the tetrahedron, which is, therefore, minimal among the cutting blocking sets in $\PG(k-1,q)$ containing an isomorphic copy of a cutting blocking set of $\PG(i,q)$ for each $i\leq k-2$. Note that its cardinality is $\sim \frac{1}{2}qk^2$ for $k$ large.

All of this seems to suggest that in order to
obtain cutting blocking sets in $\PG(k-1,q)$ of size $m(k,q)$ (or at least linear in $k$) one should look at sets that do not contain (isomorphic copies of) smaller cutting blocking sets.
\end{remark}

\subsection{Explicit Constructions of Short Minimal Codes}

In this final subsection we combine the results obtained so far to construct minimal codes of short length. To our best knowledge, this constructions produce the shortest known minimal codes, for infinitely many  dimensions and field sizes. In particular, the construction applies to all those pairs $(k,q)$ for which the rational normal tangent set of~\cite{fancsali2014lines} cannot be constructed in $\PG(k-1,q)$.

\begin{construction}\label{construction_even} Assume that $k=2t$, for some $t \in \N_{\ge 1}$. We use the construction from Theorem~\ref{thm:ConstructionGeneralSpreads}, selecting $t^2=\frac{k^2}{4}$ disjoint lines from a  linespread.  The union of these $t^2$ lines a cutting blocking set in $\PG(k-1,q)$, and we  denote the corresponding code by $\mC_{k,q}$. 
\end{construction}

\begin{proposition}\label{prop:minimalLinesCode}
 The  code $\mC_{k,q}$  of Construction~\ref{construction_even} is a minimal $[(q+1)\frac{k^2}{2},k,q(k-1)]_q$  code.
\end{proposition}

\begin{proof}
 The minimality of $\mC_{k,q}$ trivially follows from the fact that the associated projective system is a cutting blocking set; see Theorem~\ref{thm:ConstructionGeneralSpreads}. The length of the code $\mC_{k,q}$ coincides with the cardinality of the cutting blocking set, which is $(q+1)\frac{k^2}{4}$. Therefore it remains to show that $d=q(k-1)$.
 By the correspondence between projective systems and linear codes and Definition~\ref{def:projsystem}, we have that $d=n-s=(q+1){t^2}-s$, where $k=2t$ and
 $$s:=\max\{|\bar{H}\cap \bar\mT|  \, : \,  \bar{H} \subseteq  \PG(k-1,q), \,  \dim(\bar{H})=k-2 \},$$
 where $\bar{\mT}$ is the projectivization of the set $\mT$ defined in Theorem~\ref{thm:ConstructionGeneralSpreads}. We switch to vector notation and let $\mA_{\ell,m}=\{\bar\varphi(P_\ell),\bar\varphi(P_m), \bar\varphi(Q_{\ell,m,i_{\ell,m}}), \bar\varphi(Q_{\ell,m,j_{\ell,m}}) \}$ for all $1\leq \ell < m \leq t$. Let $H$ be a hyperplane of $\Fq^{k}=\Fq^{2t}$. Define the set $H_{\mT}:=\{i : \bar\varphi(P_i) \subseteq H\}$ and the integers 
 $a:=|H_{\mT}|$ and $a_{\ell,m}:=|\{ A \in  \mA_{\ell,m} : A \subseteq H\}|$ for $1\leq \ell <m \leq t$. Moreover, let $b$ denote the number of  lines forming $\bar\mT$ that are fully contained in the projectivization $\bar{H}$ of $H$. Since each of the lines forming $\bar\mT$ either intersects $\bar{H}$ in a point, or it is contained in $\bar{H}$, we have 
 \begin{equation}\label{eq:s} s=(q+1)b+t^2-b=qb+t^2.
 \end{equation}
 Therefore, finding the maximum of $s$ is the same as finding the maximum value of $b$. Now observe that $a$ cannot be equal to $t$, as otherwise $H$ would contain a basis of $\Fq^{2t}$. Moreover, we have that 
 $a_{\ell,m}\in \{0,1,4\}$. Indeed, by construction, any two subspaces in $\mA_{\ell,m}$ span the same $4$-dimensional subspace, and if $H$ contains two of them, then it contains all of them. It is readily seen that we have 
 \begin{align*}b&=a+ \sum_{\substack{\ell, m \in H_{\mT}, \\ \ell <m}}(a_{\ell,m}-2) + \sum_{\substack{\ell\in H_{\mT}, m \notin H_{\mT},\\ \ell <m}}(a_{\ell,m}-1) + \sum_{\substack{\ell\notin H_{\mT}, m \in H_{\mT},\\ \ell <m}}(a_{\ell,m}-1) + \sum_{\substack{\ell, m \notin H_{\mT},\\ \ell <m}}(a_{\ell,m}) \\
 &= a +  \sum_{\substack{\ell, m \in H_{\mT},\\ \ell <m}}2 + \sum_{\substack{\ell, m \notin H_{\mT},\\ \ell <m}}(a_{\ell,m}) \leq a+ \binom{a}{2}+\binom{t-a}{2}= a^2+\binom{t-a}{2}=:f_t(a),
 \end{align*}
 where the second equality  and the inequality both follow from the fact that   $a_{\ell,m}$ can only be equal to $0,1$ or $4$. 
 The function $f_t$ is a quadratic polynomial in $a$ with second derivative equal to $3>0$. Hence, the maximum in the interval $[0,t-1]$ is attained in one of the two interval extremes. One can see that this happens when $a=t-1$, from which $b\leq (t-1)^2$. Finally, combining this with \eqref{eq:s} we have $s=qb+t^2\leq q(t-1)^2+t^2=(q+1)t^2-q(2t-1)$ and
 $d \geq (q+1)t^2-s=q(2t-1)=q(k-1)$.
 
 On the other hand, we can take any hyperplane $H'$ containing $\bar\varphi(P_i)$, for each $i\in [t-1]$. The projectivization of such a hyperplane contains exactly $b=(t-1)^2$ lines forming $\bar\mT$, and therefore 
 $n-|\bar{H'}\cap \bar\mT|=q(k-1)$.
\end{proof}

\begin{example}
 Let $k=6$ and take the cutting blocking set  obtained in Example~\ref{exa:cutting_from_spread}. This is a cutting blocking set arising from Construction~\ref{construction_even}.
 When $q=2$, we take $\gamma$ to be a root of $x^2+x+1$ and obtain a minimal $[27,6]_2$ code $\mC_{6,2}$ whose generator matrix is
 $$\left(\begin{array}{ccccccccccccccccccccccccccc}
 1 & 1 & 0 & 0 & 0 & 0 & 0 & 0 & 0 & 1 & 1 & 0 & 1 & 1 & 0 & 0 & 0 & 0 & 1 & 0 & 1 & 1 & 0 & 1 & 0 & 0 & 0  \\
 0 & 1 & 1 & 0 & 0 & 0 & 0 & 0 & 0 & 0 & 1 & 1 & 0 & 1 & 1 & 0 & 0 & 0 & 0 & 1 & 1 & 0 & 1 & 1 & 0 & 0 & 0   \\
 0 & 0 & 0 & 1 & 1 & 0 & 0 & 0 & 0 & 1 & 1 & 0 & 0 & 0 & 0 & 1 & 1 & 0 & 0 & 1 & 1 & 0 & 0 & 0 & 1 & 0 & 1 \\
 0 & 0 & 0 & 0 & 1 & 1 & 0 & 0 & 0 & 0 & 1 & 1 & 0 & 0 & 0 & 0 & 1 & 1 & 1 & 1 & 0 & 0 & 0 & 0 & 0 & 1 & 1  \\
 0 & 0 & 0 & 0 & 0 & 0 & 1 & 1 & 0 & 0 & 0 & 0 & 1 & 1 & 0 & 1 & 1 & 0 & 0 & 0 & 0 & 0 & 1 & 1 & 0 & 1 & 1 \\
 0 & 0 & 0 & 0 & 0 & 0 & 0 & 1 & 1 & 0 & 0 & 0 & 0 & 1 & 1 & 0 & 1 & 1 & 0 & 0 & 0 & 1 & 1 & 0 & 1 & 1 & 0  
 \end{array} \right).$$
\end{example}

Our second construction combines Theorem~\ref{thm:ConstructionGeneralSpreads} with the concept of a Baer subplane. 

\begin{construction}\label{construction_baer} Assume that $k=3t$ for some $t \in \N$ and that $q$ is a square. We first use the construction from Theorem~\ref{thm:ConstructionGeneralSpreads} by selecting $t^2=\frac{k^2}{9}$ disjoint planes from a $2$-spread. Then
we choose two disjoint Baer subplanes in each of these planes.  The union of the selected $2t^2$ Baer subplanes is a cutting blocking set in $\PG(k-1,q)$, and we denote the corresponding code by~$\mD_{k,q}$.
\end{construction}

\begin{proposition}\label{prop:minimalBaerCode}
 The  code $\mD_{k,q}$ of Construction~\ref{construction_baer} is a minimal $[(q+\sqrt{q}+1)\frac{2k^2}{9},k, d]_q$ code, where $d\geq q\big(\frac{4}{3}k-2\big)$.
\end{proposition}

\begin{proof}
 The minimality of $\mD_{k,q}$ trivially follows from the fact that the associated projective system~$\mB$ is a cutting blocking set (Theorem~\ref{thm:ConstructionGeneralSpreads} and Proposition~\ref{prop:inductiveConstruction}). The length of the code~$\mD_{k,q}$ coincides with the cardinality of the cutting blocking set, which is $(q+\sqrt{q}+1)\frac{2k^2}{9}$. We only need to prove that $d\geq q\big(\frac{4}{3}k-2\big)$. We let $k=3t$ and proceed as before, finding an upper bound on  $$s:=\max\{|\bar{H}\cap \mB|  \, : \,    \bar{H} \subseteq  \PG(k-1,q), \, \dim(\bar{H})=k-2 \}.$$
 Observe that $\mB$ is obtained by first forming the cutting blocking set $\bar\mT$ as in Theorem~\ref{thm:ConstructionGeneralSpreads}, which is the union of $t^2$ planes $\Lambda_1,\ldots, \Lambda_{t^2}$, and then selecting two disjoint Baer subplanes $\mB_{i,1}, \mB_{i,2}$ in each~$\Lambda_i$. 
 Let $\bar{H}$ be a hyperplane in $\PG(k-1,q)$ and let $b$ denote the number of planes $\Lambda_i$ that are fully contained in $\bar{H}$.  With this notation, we have
 \begin{align}\label{eq:inequaity_baer_code} |\mB\cap \bar{H}| &= 2b(q+\sqrt{q}+1)+ \sum_{i \, : \, \Lambda_i \not\subseteq \bar{H}} |\mB_{i,1}\cap \bar{H}|+|\mB_{i,2}\cap \bar{H}| \nonumber \\ 
 &\leq 2(q+\sqrt{q}+1)b+ 2(\sqrt{q}+1)(t^2-b),\end{align}
 where the last inequality follows from the fact that a hyperplane $\bar{H}$ meets a  Baer subplane in either $1$, $\sqrt{q}+1$ or $q+\sqrt{q}+1$ points.
 Moreover, arguing as in the proof of Proposition~\ref{prop:minimalLinesCode}, one proves that $b\leq (t-1)^2$. Combining this with \eqref{eq:inequaity_baer_code} we obtain that 
 $s\leq 2(q+\sqrt{q}+1)t^2-2q(2t-1)$ and finally $d=n-s\geq q\big(\frac{4}{3}k-2\big)$.
\end{proof}

We conclude with a remark that summarizes the code lengths obtained from the constructions and results of this section. 

\begin{remark}
 For every positive integer $k$ and every prime power $q$, we have provided explicit constructions of minimal $[n_{k,q},k]_q$ codes with
 $$ n_{k,q} = \begin{cases}\frac{1}{4}(q+1)k^2 & \mbox{ if } k \equiv 0 \mod 2,  \\
 \frac{2}{9}(q+\sqrt{q}+1)k^2 & \mbox{ if } k \equiv 0 \mod 3 \mbox{ and } q \mbox{ is a square, } \\
  \frac{2}{9}(q+\sqrt{q}+1)(k-1)^2+(q-1)(k-1)+1 & \mbox{ if } k \equiv 1 \mod 3 \mbox{ and } q \mbox{ is a square, } \\
 \frac{1}{4}(q+1)(k+1)^2-(2k+q-2) & \mbox{ otherwise. }
 \end{cases} $$
 The first length is given by Construction~\ref{construction_even}, the second length is given by Construction~\ref{construction_baer}, and the last two lengths are obtained by combining  Proposition~\ref{prop:inductivek+1} with these two constructions. It is easy to see that the minimum distance $d$ of any code obtained using  Proposition~\ref{prop:inductivek+1} meets the bound of Theorem~\ref{thm:lower_bound_old_conjecture} with equality, i.e., $d=(q-1)(k-1)+1$.  
\end{remark}

\bigskip

\bigskip

\bibliographystyle{abbrv}
\bibliography{biblio.bib}

\end{document}